\documentclass{article}
\usepackage[table]{xcolor}
\usepackage{graphicx}
\usepackage{amsmath,amsfonts,amssymb,amsthm}
\usepackage{graphicx}
\usepackage{multirow}
\usepackage{tikz,pgf}
\usepackage{url}
\usepackage{amsmath}
\usepackage{graphicx}
\usepackage{epsfig}
\usepackage{epstopdf}
\usepackage{color}
\usepackage{xcolor}
\usepackage[shortlabels]{enumitem}
\usepackage{comment}

\def\tto{\;{\lower 1pt \hbox{$\rightarrow$}}\kern -10pt
	\hbox{\raise 2pt \hbox{$\rightarrow$}}\;}
\def\Hat{\widehat}

\def\Bar{\overline}
\def\ra{\rangle}
\def\la{\langle}

\def\epsilon{\varepsilon}

\def\h{\hfill\Box}
\def\R{\Bbb R}

\def\N{\Bbb N}
\def\ox{\bar{x}}

\def\co{\mbox{\rm co}}

\def\ri{\mbox{\rm ri}}

\def\aff{\mbox{\rm aff}}
\def\epi{\mbox{\rm epi}}

\def\dim{\mbox{\rm dim}}
\def\dom{\mbox{\rm dom}}
\def\aff{\mbox{\rm aff}}

\def\cone{\mbox{\rm cone}}

\def\h{\hfill\square}

\def\ph{\varphi}

\def\oR{\Bar{\R}}

\def\al{\alpha}

\def\ph{\varphi}

\def\oR{\Bar{\R}}

\def\al{\alpha}

\DeclareMathOperator{\qri}{qri}
\DeclareMathOperator{\prox}{prox}

\DeclareMathOperator{\argmin}{arg\,min}

\def\Ts{\textstyle}
\setlist[enumerate,1]{itemsep=0.0ex,parsep=0.5ex,label={\rm(\alph*)},leftmargin=*, align=left}
\newcounter{lk}

\usepackage[colorlinks=true]{hyperref}

\definecolor{gls}{rgb}{0.0, 0.5, 1.0}

\newtheorem{Theorem}{Theorem}[section]
\newtheorem{Proposition}[Theorem]{Proposition}
\newtheorem{Remark}[Theorem]{Remark}
\newtheorem{Lemma}[Theorem]{Lemma}
\newtheorem{Corollary}[Theorem]{Corollary}
\newtheorem{Definition}[Theorem]{Definition}
\newtheorem{Example}[Theorem]{Example}

\numberwithin{equation}{section}

\title{Lagrange Multipliers and Duality   with Applications to Constrained  Support Vector Machine}

\author{{\sc Nguyen Mau   Nam}\footnote{Fariborz Maseeh Department of Mathematics and Statistics, Portland State University, Portland, OR
97207, USA (mnn3@pdx.edu). Research of this author was partly supported by the National Science
Foundation (NSF): NSF-RTG grant No.  DMS-2136228.},
{\sc Gary Sandine}\footnote{Fariborz Maseeh Department of Mathematics and Statistics, Portland State University, Portland, OR
97207, USA (gsandine@pdx.edu).}, and 
{\sc Quoc Tran-Dinh}\footnote{Department of Statistics and Operations Research,
The University of North Carolina at Chapel Hill, Chapel Hill, NC 27599 (quoctd@email.unc.edu). Research of this author was partly supported by the National Science
Foundation (NSF): NSF-RTG grant No. NSF DMS-2134107 and the Office of Naval
Research (ONR), grant No. N00014-23-1-2588.}
}

\date{}  

\begin{document}

\maketitle

\begin{center}
\bf In the Cherished Memory of Prof.~Dr.~Franco Giannessi
\end{center}

\textbf{Abstract.} In this paper, we employ the concept of quasi-relative interior to analyze the method of Lagrange multipliers and establish strong Lagrangian duality for nonsmooth convex optimization problems in Hilbert spaces. Then, we generalize the classical support vector machine (SVM) model by incorporating a new geometric constraint or a regularizer on the separating hyperplane, serving as a regularization mechanism for the SVM model. This new SVM model is examined using Lagrangian duality and other convex optimization techniques in both theoretical and numerical aspects via a new subgradient algorithm as well as a primal-dual method.


\section{Introduction}
The method of Lagrange multipliers lies at the heart of optimization, forming the foundation of various modern optimization algorithms. 
This method was introduced at the end of the 18th century by Joseph-Louis Lagrange as a 
way to tackle constrained optimization problems and has been extensively 
studied and extended through the 20th century. 
Therefore, it plays a critical role in a multitude of 
applications spanning several disciplines, including economics and operations research, engineering, and modern machine learning.
In a nutshell, the method of Lagrange multipliers converts a 
constrained problem into an unconstrained problem through the introduction of 
Lagrange multipliers, which capture the relationship between the gradient of 
the objective function and the gradients of the constraints. The modern framework of Lagrangian duality appeared much later, particularly in the mid-20th century, with significant contributions made by a number of researchers, such as John von Neumann and George Dantzig, particularly in linear programming. Von Neumann's work on game theory and duality laid foundational concepts, while Dantzig's simplex method highlighted the practical utility of duality in solving linear programming problems. Lagrangian duality provides a framework to address constrained optimization problems by constructing their dual problem through the Lagrange function. In many cases, dual problems serve as valuable tools for analyzing the primal problems, both theoretically and numerically. 
The method of Lagrange multipliers and Lagrangian duality are closely related, as Lagrangian duality is built based on the method of Lagrange multipliers.
For foundational theory on convex analysis, the method of Lagrange multipliers, and Lagrangian duality, we refer to~\cite{bc,Bert,Borwein2000,boyd-2004-convex,giannessi-2007-theory,bmn,bmn2022,r,Zalinescu2002}. In particular, the work of Giannessi~\cite{giannessi-2007-theory} plays a crucial role in the modern theory of Lagrangian duality.

The first part of this paper is devoted to a comprehensive study of the method of Lagrange multipliers and Lagrangian duality for nonsmooth convex constrained optimization in infinite-dimensional spaces. We introduce new Slater conditions using the notion of the quasi-relative interior introduced by Borwein and Lewis \cite{bl-qri}. These quasi-relative interior Slater conditions turn out to be useful in obtaining both the generalized Karush–Kuhn–Tucker (KKT) conditions and strong Lagrangian duality for an important class of convex programs in infinite-dimensional spaces.

In the second part of this paper, we generalize the classical support vector machine (SVM) problem to a new model in which we impose a constraint involving the normal, or weight vector $w$ and the bias $b$ of the separating hyperplane $\langle w, x \rangle +b=0$. The motivation comes from the seminal work \cite{Pegasos}, introducing the \emph{Pegasos} algorithm, which incorporates an optional Euclidean projection onto a ball centered at the origin to enhance its performance.  By introducing a constraint on the support vector and the bias of the separating hyperplane, our new model not only provides the mathematical foundation for the optional projection step as in \emph{Pegasos} \cite{Pegasos}, but also allows us to incorporate different ways of controlling the weight vector and the bias as a regularization mechanism in SVM.

We also provide a detailed formulation and our approach to studying this new constrained SVM model from both theoretical and numerical perspectives, which offer new insights into the classical SVM technique. Using tools from convex analysis and optimization, we study the existence and uniqueness of optimal solutions to the new constrained SVM problem for both the hard-margin and the soft-margin settings, as well as their unconstrained models based on the penalty method and their dual problems obtained by the Lagrangian duality theory mentioned earlier. It turns out that the Lagrange dual problems obtained contain a term involving the squared distance function to the constraint set. To the best of our knowledge,  this has not been considered in the existing literature. We employ two distinct approaches to numerically solve the constrained SVM problem. The first approach uses the projected subgradient/stochastic subgradient method, which generalizes the existing method from \cite{Pegasos}. 
Unlike existing subgradient methods, e.g.,  \cite{Pegasos}, our new algorithm works as a hybrid scheme combining both subgradient and gradient steps, and removing the projection on a compact set, while maintaining the same convergence rate as in \cite{Pegasos}.
Alternatively, based on the differentiability of the squared distance function, we propose another approach that involves solving the Lagrange dual problem and using its approximate solutions to construct an approximate optimal solution to the primal SVM problem.

The remainder of this paper is organized as follows. Section 2 recalls 
necessary tools from convex analysis that will be used in this work.  Section 
3 examines Lagrange multipliers and duality for constrained optimization 
problems under the Slater condition, specifically involving the quasi-relative 
interior and conditions for strong duality. Section 4 focuses on developing 
Lagrangian primal and dual solutions for hard-margin and soft-margin 
constrained support vector machine (SVM) problems in Hilbert spaces. It also 
presents a subgradient method using convex analysis to approximate a solution 
of the new SVM model. This method can be viewed as an illustration of our 
theoretical results. Section 4 concludes with numerical examples illustrating 
the usage of Lagrangian duality to solve constrained hard-margin and 
soft-margin SVM problems.

Throughout the paper, we use notation and concepts of convex analysis as in~\cite{bmn,bmn2022}. For simplicity of presentation, we only consider the Hilbert space setting, although several results hold in locally convex topological vector spaces. Consider a real Hilbert space $\mathcal H$ that is nontrivial, i.e., $\mathcal H\neq \{0\}$. The inner product of two elements $u, v\in \mathcal H$ is denoted by $\la u, v\ra$. For $A\subset  \mathcal H$, the cone generated by $A$ is the set $\cone(A) =\{ta\; |\; t\geq 0, a\in A\}$. The convex hull of $A$ is the set $\co(A)=\bigcap\{C\; |\; C \; \mbox{\rm is a convex set containing }A\}$, and the affine hull of $A$ is the set $\aff(A)=\bigcap\{M\; |\; M \; \mbox{\rm is an affine set containing }A\}$. The set $\oR=\R\cup \{\infty\}$ is the extended real line. The set $\R^m_+$ stands for the nonnegative orthant. The open ball (resp., closed ball) with center $a\in \mathcal H$ and radius $r>0$ is denoted by $B(a; r)$ (resp., $B^\prime(a; r)$).

\section{Mathematical Tools from Convex Analysis}\label{sec:background}

\setcounter{equation}{0}

In this section, we will review some concepts and prove necessary results in convex optimization, such as the Slater condition, Fenchel duality, and Lagrangian duality, which will be used in this paper.
The readers are referred to \cite{bc,Bert,Borwein2000,boyd-2004-convex,bmn,bmn2022,r,Zalinescu2002} for more details.


\subsection{Basic definitions and preliminaries}

A subset $\Omega$ of $\mathcal H$ is said to be convex if $\lambda x+(1-\lambda)y\in \Omega$ whenever $x, y\in \Omega$ and $0<\lambda<1$. 

Given a nonempty convex set $\Omega$ in $\mathcal H$ and $w_0\in \Omega$, the normal cone of $\Omega$ at $w_0$ is defined as
\begin{equation*}
N(w_0; \Omega)=\big\{v\in \mathcal H\; \big |\; \la v, w-w_0\ra\leq 0\ \; \mbox{\rm for all }w\in \Omega\}.
\end{equation*}
We set $N(w_0; \Omega)=\emptyset$ if $w_0\notin \Omega$.

Given an extended-real-valued function $f\colon \mathcal H\to \oR$, the effective domain of $f$ is defined by
\begin{equation*}
    \dom(f)=\big\{x\in \mathcal H\; \big |\; f(x)<\infty\big\}.
\end{equation*}
The epigraph of $f$ is the following subset of $\mathcal H\times \R$:
\begin{equation*}
    \epi(f)=\big\{(w, \lambda)\in \mathcal H\times \R\; \big | \; f(w)\leq \lambda\big\}.
\end{equation*}
Note that $\mathcal H\times \R$ is also a Hilbert space with the inner product:
\begin{equation*}
    \la (w_1, \lambda_1), (w_2, \lambda_2)\ra=\la w_1, w_2\ra+\lambda_1\lambda_2, \ \; (w_1, \lambda_1), (w_2, \lambda_2)\in \mathcal H\times \R.
\end{equation*}
We say that $f$ is a convex function if 
\begin{equation*}
    f(\lambda x+(1-\lambda)y)\leq \lambda f(x)+(1-\lambda)f(y)\ \; \mbox{\rm for all }x, y \in \mathcal H\; \mbox{\rm and }0<\lambda<1.
\end{equation*}
It follows from the definition that $f$ is a convex function if and only if $\epi(f)$ is a convex subset of $\mathcal H\times \R$.

Given a convex function $f\colon \mathcal H\to \oR$ and $w_0 \in \dom(f)$, an element $v\in \mathcal H$ is called a subgradient of $f$ at $w_0$ if
\begin{equation*}
    \la v, w-w_0\ra\leq f(w)-f(w_0)\ \; \mbox{\rm for all }w\in \mathcal H. 
\end{equation*}
The collection of all subgradients of $f$ at $w_0$ is called the subdifferential of $f$ at this point and is denoted by $\partial f(w_0)$. We set $\partial f(w_0)=\emptyset$ if $w_0\notin \dom(f)$. 

The subdifferential of a convex function $f$ at $w_0\in \dom(f)$ can be represented in terms of  the normal cone to the epigraph of $f$ by
\begin{equation}\label{SNR}
    \partial f(w_0)=\big\{v\in \mathcal H\; \big |\; (v, -1)\in N((w_0, f(w_0)); \epi(f))\big\}.
\end{equation}
Furthermore, we have the relationship:
\begin{equation*}
    N(w_0; \dom(f))=\big\{v \in \mathcal H\; \big |\; (v, 0)\in N((w_0, f(w_0)); \epi(f))\big\}.
\end{equation*}
An important example of an extended-real-valued function is the indicator function associated with a subset $\Omega$ of $\mathcal H$ given by
\begin{equation*}
    \delta(x; \Omega)=\begin{cases}0 &\mbox{\rm if }x\in \Omega, \\
    \infty &\mbox{\rm if }x\notin \Omega.        
    \end{cases}
\end{equation*}
If $\Omega$ is a nonempty convex set, we can easily show that
\begin{equation*}
    \partial \delta(w_0; \Omega)=N(w_0; \Omega)\ \; \mbox{\rm for every }w_0\in \mathcal H.
\end{equation*}

In the case where $\mathcal H$ is the Euclidean space $\R^n$, an important generalized interior notion called the relative interior of a set $\Omega$ is defined by
\begin{equation*}
    \ri(\Omega)=\big\{ x\in \Omega\; \big |\; \exists \delta>0\; \mbox{\rm such that }B(x; \delta)\cap \aff(\Omega)\subset\Omega\big\}.
\end{equation*}
This means that the relative interior of $\Omega$ is its interior within $\aff(\Omega)$, the affine hull of $\Omega$, with the subspace topology. 

It turns out that the interior and relative interior play a crucial role in developing calculus rules for normal cones to convex sets and subdifferentials in finite dimensions; see, e.g., \cite{bmn2022}.
\begin{Theorem}\label{NIR}
Let $\Omega_1$ and $\Omega_2$ be two nonempty convex sets in a Hilbert space $\mathcal H$. Suppose that one of the following conditions is satisfied:
\begin{enumerate}
    \item $\mbox{\rm int}(\Omega_1)\cap \Omega_2\neq\emptyset$,
    \item $\mathcal H=\R^n$ and $\ri(\Omega_1)\cap \ri(\Omega_2)\neq\emptyset$.
\end{enumerate}
Then we have 
\begin{equation*}
    N(w_0; \Omega_1\cap \Omega_2)=N(w_0; \Omega_1)+N(w_0; \Omega_2)\;\ \mbox{\rm for every }w_0\in \Omega_1\cap \Omega_2.
\end{equation*}
\end{Theorem}
Considerable effort has been devoted to developing new concepts of generalized relative interior in infinite-dimensional spaces. Among these concepts, the notion of quasi-relative interior (see \cite{bl-qri}) is a successful one.  Let $\Omega$ be a subset of a Hilbert space $\mathcal H$. The quasi-relative interior of $\Omega$ is defined by
\begin{equation*}
    \qri(\Omega)=\big\{x\in \Omega\; \big |\; \overline{\cone(\Omega-x)}\; \mbox{\rm is a linear subspace}\big\}.
\end{equation*}
The quasi-relative interior coincides with the relative interior when 
$\dim(\mathcal H)<\infty$, whereas in infinite-dimensional settings, it often serves a role analogous to that of the relative interior.

A subset $P$ of $\mathcal H$ is called a polyhedral convex set if there exist $v_i\in \mathcal H$ and $b_i\in \R$ for $i=1, \ldots, m$ such that
\begin{equation*}
    P=\big\{w\in \mathcal H\; \big |\; \la v_i, x\ra\leq b_i\; \ \mbox{\rm for all }i=1, \ldots, m\big\}.
\end{equation*}
The quasi-relative interior can be used to obtain an improved version of Theorem \ref{NIR} in the case where one of the sets under consideration is polyhedral; see \cite[Theorem~3.87]{bmn2022}. 
\begin{Theorem}\label{PNIR} Let $\Omega$ be a convex set, and let $P$ be a  polyhedral convex set in a Hilbert space $\mathcal H$ such that $\qri(\Omega)\cap P\neq\emptyset$. Then
\begin{equation*}
    N(w_0; \Omega\cap P)=N(w_0; \Omega)+N(w_0; P)\ \; \mbox{\rm for every }w_0\in \Omega\cap P.
\end{equation*}
\end{Theorem}

To continue, we recall the well-known definitions of distance function and Euclidean projection in a Hilbert space $\mathcal H$. Let $\Omega$ be a nonempty closed convex set in a Hilbert space. Given $x\in \mathcal H$, define the distance from $x$ to $\Omega$ by
\begin{equation*}
    d(x; \Omega)=\inf\big\{\|x-w\|\; \big |\; w\in \Omega\big\}
\end{equation*}
and the Euclidean projection from $x$ to $\Omega$ by
\begin{equation*}
    \mathcal P(x; \Omega)=\big\{w\in \Omega\; \big|\; d(x; \Omega)=\|x-w\|\big\} = \mathrm{arg}\min\big\{ \|x - w\| \mid w \in \Omega\big\}.
\end{equation*}
It is well-known that the distance function \(d(\cdot; \Omega)\) is nonexpansive and convex, and that \(\mathcal{P}(x; \Omega)\) is a singleton in this setting.

Finally, we recall a useful result on the Fr\'echet differentiability of the squared distance function. 

\begin{Theorem} Let $\Omega$ be a nonempty closed convex set in a Hilbert space $\mathcal H$. Define the function $\ph(\cdot) = \frac{1}{2} \big(d(\cdot; \Omega)\big)^2$ on $\mathcal H$. Then $\ph$ is Fr\'echet differentiable and its derivative is 
\begin{equation*}
    \nabla \ph(w)= w-\mathcal P(w; \Omega) \ \; \mbox{\rm for every }w\in \mathcal H.
\end{equation*}
\end{Theorem}

\subsection{Subgradients of maximum functions}

In this subsection, we obtain an upper estimate for the subdifferential of the maximum of two convex functions, one of which is not necessarily continuous. This upper estimate will be used to study the method of Lagrange multipliers for \eqref{eq:cvx_opt} as well as for \eqref{Q2}.

The lemma below states a simple, known result, which is useful for proving the subsequent theorem. 

\begin{Lemma}\label{Nepi} Let $f\colon \mathcal H\to \oR$ be a proper convex function. Fix any $w_0\in \dom(f)$ and $\lambda_0\in \R$ such that $\lambda_0\geq f(w_0)$. Then
\begin{equation}\label{epiinc}
    N((w_0, \lambda_0); \epi(f))\subset N((w_0, f(w_0)); \epi(f)). 
\end{equation}
\end{Lemma}
\begin{proof} 
Take any $(v, -\gamma)\in N((w_0, \lambda_0); \epi(f))$. By the definition of normal cone, we have
\begin{equation*}
    \la v, w-w_0\ra-\gamma(\lambda-\lambda_0)\leq 0\; \ \mbox{\rm for all }(w, \lambda)\in \epi(f). 
\end{equation*}
Using this inequality with $w=w_0$ and $\lambda=\lambda_0+1$ yields $\gamma\geq 0$.  Since $\lambda_0\geq f(w_0)$, we have
\begin{equation*}
  \la v, w-w_0\ra-\gamma(\lambda-f(w_0)) \leq  \la v, w-w_0\ra-\gamma(\lambda-\lambda_0)\leq 0\; \ \mbox{\rm for all }(w, \lambda)\in \epi(f).   
\end{equation*}
This inequality shows that $(v, -\gamma)\in N((w_0, f(w_0)); \epi(f))$, which implies \eqref{epiinc}.
\end{proof}

The theorem below provides an upper estimate for the subdifferential of the maximum of two convex functions. Unlike the existing result in \cite[Theorem~3.59]{bmn2022}, our approach requires the continuity of only one of the underlying functions.

\begin{Theorem}\label{MaxRule} Let $g\colon \mathcal H\to \R$ be a continuous convex function, and let $h\colon \mathcal H\to \oR$ be a proper convex function. For $w\in\mathcal{H}$, define
\begin{equation*}
    f(w)=\max\{g(w), h(w)\},
\end{equation*}
and suppose that $\partial h(w_0)\neq\emptyset$. Then, we always have
\begin{equation}\label{maxus}
\partial f(w_0)\subset \mbox{\rm co}\big(\partial g(w_0)\cup \partial h(w_0)\big)+N(w_0; \dom(h)).
\end{equation}
The reverse inclusion holds if $g(w_0)=h(w_0)$. 
\end{Theorem}

\begin{proof} 
First, it is obvious that
\begin{equation*}
    \epi(f)=\epi(g)\cap \epi(h).
\end{equation*}
Next, since $g$ is continuous, by  \cite[Corollary~2.145]{bmn2022} or a direct proof, we have
\begin{equation*}
    \mbox{\rm int}(\epi(g))=\big\{(x, \lambda)\in \mathcal H\times \R\; \big |\; g(x)<\lambda\big\}.
\end{equation*}
It follows that
\begin{equation*}
    \mbox{\rm int}(\epi(g))\cap \epi(h)\neq \emptyset
\end{equation*}
because if $\lambda_0>\max\{g(w_0), h(w_0)\}$, then $(w_0, \lambda_0)$ belongs to this intersection. 

Now, since $f(w_0)\geq g(w_0)$ and $f(w_0)\geq h(w_0)$,  we can use Theorem \ref{NIR} and Lemma \ref{Nepi} to obtain
\begin{equation*}
\begin{aligned}
    N((w_0, f(w_0)); \epi(f))&=N((w_0, f(w_0)); \epi(g))+N((w_0, f(w_0)); \epi(h))\\
    &\subset N((w_0, g(w_0)); \epi(g))+N((w_0, h(w_0)); \epi(h)).
    \end{aligned}
\end{equation*}
For  any $v\in \partial f(w_0)$, by \eqref{SNR}, we get  $(v, -1)\in N((w_0, f(w_0)); \epi(f))$. Thus we have the representation
\begin{equation*}
    (v, -1)=(v_1, -\gamma_1)+(v_2, -\gamma_2),
\end{equation*}
where $(v_1, -\gamma_1)\in N((w_0, g(w_0)); \epi(g))$ and $(v_2, -\gamma_2)\in N((w_0, h(w_0)); \epi(h))$. We can see that $\gamma_i\geq 0$ for $i=1, 2$ as in the proof of Lemma \ref{Nepi}. If $\gamma_1>0$, then $(v_1/\gamma_1, -1)\in N((w_0, g(w_0)); \epi(g))$, so $v_1\in \gamma_1\partial g(w_0)$ by \eqref{SNR}.  By the continuity of $g$,  this inclusion also holds if $\gamma_1=0$; see \cite[Corollary 3.38]{bmn2022} because $v_1=0$ in this case. 
If $\gamma_2>0$, then $v_2\in \gamma_2\partial h(w_0)$, and if $\gamma_2=0$, then $v_2\in N(w_0; \dom(h))$ by \cite[Theorem~3.37]{bmn2022}. Thus, we always have
\begin{equation*}
    v_2\in \gamma_2\partial h(w_0)+N(w_0; \dom(h))
\end{equation*}
 because $\partial h(w_0)\neq\emptyset$ and it is obvious that $0\in N(w_0; \dom(h))$.   Since $\gamma_1+\gamma_2=1$, we conclude that
 \begin{equation*}
     v\in \gamma_1\partial g(w_0)+\gamma_2\partial h(w_0)+N(w_0; \dom(h))\subset \mbox{\rm co}\big(\partial g(w_0)\cup \partial h(w_0)\big)+N(w_0; \dom h),
 \end{equation*}
 which justifies \eqref{maxus}.
 
 Finally, the proof of the reverse inclusion under the assumption  $g(w_0)=h(w_0)$ is straightforward and is left for the reader. 
\end{proof}

\section{Lagrange Multipliers and  Duality via Quasi-Relative Interiors}

\setcounter{equation}{0}

\subsection{Problem statement and motivation}\label{subsec:prob_stat}
Let $\mathcal{H}$ be a Hilbert space. 
We are interested in the following convex optimization problem:
\begin{equation}\label{eq:cvx_opt}
\left\{\begin{array}{ll}
{\displaystyle\min_{ w \in \mathcal{H}}} & \phi(w) = f(w) + h(w), \vspace{1ex}\\
\mathrm{s.t.} &  g_i(w) \leq 0\; \mbox{\rm for }i=1, \ldots,m,
\end{array}\right.
\end{equation}
where $f \colon  \mathcal{H} \to \mathbb{R}$ is a continuous convex function, $h \colon  \mathcal{H}  \to \oR$ is a proper convex function, and $g_i \colon {\mathcal H} \to \R$ are continuous convex functions. 
Here, we assume that $h$ is a proper convex function to cover different settings such as regularizers, penalty terms, and simple constraints. 
The term $f$ is usually used to capture a data fidelity term or a loss function for concrete applications in machine learning, statistical learning, and data science. 

Alternatively, we also consider a variation of \eqref{eq:cvx_opt} of the following form:
\begin{equation}\label{Q2}
\left\{\begin{array}{ll}
{\displaystyle\min_{ w \in \mathcal{H}}} &f(w), \vspace{1ex}\\
\mathrm{s.t.} &  g_i(w) \leq 0\; \mbox{\rm for }i=1, \ldots,m, \vspace{1ex}\\
&\mathcal Aw=b, \vspace{1ex}\\
&w\in \Theta,
\end{array}\right.
\end{equation}
where $\mathcal{A}\colon \mathcal H\to \R^q$ is a continuous linear mapping, 
$b\in\R^q$, and $\Theta\subset \mathcal H$ is a nonempty convex set. Clearly, 
if $h$ is the indicator function for $\Theta$ in~\eqref{eq:cvx_opt}, $h(w) = 
\delta(w;\Theta)$, then we obtain \eqref{Q2} without the linear constraint.

In \eqref{eq:cvx_opt}, let us denote 
\begin{equation*}
\mathcal{F} = \big\{ w \in \dom{(\phi)}\; \big |\;  g_i(w) \leq 0\; \mbox{\rm for }i=1, \ldots, m\big\},
\end{equation*}
which is assumed to be nonempty. Then, we can rewrite \eqref{eq:cvx_opt} in the following compact form:
\begin{equation}\label{eq:cvx_opt1}
\min_{w \in \mathcal{F} } \Big\{ \phi(w) = f(w) + h(w) \Big\}.
\end{equation}
Now, let us consider some common special cases of \eqref{eq:cvx_opt}, or equivalently, \eqref{eq:cvx_opt1}.
\begin{itemize}
\item \textbf{Composite convex optimization:}
If $g_i$ for $i=1, \ldots, m$ are absent, i.e., $\mathcal{F} = \dom{(\phi)}$, then \eqref{eq:cvx_opt} reduces to $\min_w \big\{ \phi(w) = f(w) + h(w) \big\}$, which is known as composite convex minimization.

\item \textbf{Classical convex optimization model:}
If $h(w) = \delta(\cdot; \Theta)$, the indicator function for a nonempty closed and convex set $\Theta$ in $\mathcal{H}$, then \eqref{eq:cvx_opt} reduces to the classical convex optimization problem \eqref{Q2} as in various textbooks. 
This model also covers convex cone or conic programs such as linear programming, convex quadratic programming, second-order cone programming, and semi-definite programming.
\end{itemize}

Our motivation in the subsequent work comes from the support vector machine and its extensions. The support vector machine (SVM) is a classical technique in supervised machine learning, introduced by Vapnik and Chervonenkis in 1964 \cite{svm64}, which presents as one of the most fundamental methods for binary classification. 
The details of this model will be discussed in Section~\ref{svma}. 
Here, we simply write down the mathematical form of the classical soft-margin SVM problem as an optimization model:
\begin{equation}\label{eq:SVM1}
\min_{w \in \mathcal H, \, b \in \mathbb{R}} \Big\{ \phi(w, b) = \frac{1}{2}\Vert w \Vert^2 + \frac{C}{m}\sum_{i=1}^m \ell( y_i(\la \mathbf x_i, w\ra + b) ) \Big\},
\end{equation}
where $\{ (\mathbf x_i, y_i) \}_{i=1}^m$ is a given data set such that $\mathbf x_i$ is an example of the feature vector $x$ and $y_i \in \{-1, 1\}$ is the associated label,  $\ell(\tau) = \max\{0, 1 - \tau\}$, $\tau\in \R$, is the Hinge loss function, and $C > 0$ is a given penalty parameter.
Though this model is standard, it may exclude some applications where additional conditions on the weight vector $w$ and the bias term $b$ are required.
For instance, a constraint of the form $w \in \Theta$ is required, or a regularizer $\lambda R(w)$ is added to the model.
These generalizations make the classical theory and algorithms for SVM inapplicable to the target applications.
This work specifically targets the constrained problem as one example.

Another extension using matrices rather than vectors is known as support matrix machine (SMM) problems, which have been presented in several works, including \cite{luo2015}.
For instance, in \cite{kumari2025,luo2015,pan2022,pirsiavash2009}, the authors essentially studied the following problem:
\begin{equation*}\label{eq:SMM1}
\min_{W \in \mathbb{R}^{p\times q},\, b \in \mathbb{R}} \Big\{ \Phi(W, b) = \frac{1}{2}\mathrm{tr}(W^{\top} W) + \Vert W \Vert_{*} + \frac{C}{N}\sum_{i=1}^N \ell( y_i(X_i^{\top}W + b) ) \Big\},
\end{equation*}
where $\mathrm{tr}(\cdot)$ is the trace operator and $\Vert \cdot\Vert_{*}$ denotes the nuclear norm (i.e., the sum of singular values).
Clearly, this extension does not reduce to \eqref{eq:SVM1} when $q = 1$ due to the nuclear norm regularizer.

\subsection{Slater conditions and the method of Lagrange multipliers}

Our main goal in this subsection is to study the method of Lagrange multipliers for \eqref{eq:cvx_opt} and its refinement to the case of \eqref{Q2}. For simplicity, we present our results in Hilbert spaces, although these results can be easily generalized to the case of locally convex topological vector spaces. 

The lemma below is simple but useful in what follows. We provide the proof for the convenience of the reader. 

\begin{Lemma}\label{subE} 
Let $h\colon \mathcal H\to \oR$ be a proper convex function. For any $w_0\in \dom(h)$, we have
\begin{equation*}
    \partial h(w_0)=\partial h(w_0)+N(w_0; \dom(h)). 
\end{equation*}
\end{Lemma}
\begin{proof} Since $0\in N(w_0; \dom(h))$, it is obvious that
\begin{equation*}
    \partial h(w_0)\subset \partial h(w_0)+N(w_0; \dom(h)).
\end{equation*}
Conversely, take any $v\in \partial h(w_0)+N(w_0; \dom(h))$. Then we have the representation
\begin{equation*}
    v=v_1+v_2,
\end{equation*}
where $v_1\in \partial h(w_0)$ and $v_2\in N(w_0; \dom(h))$. 
Take any $w\in \dom(h)$ and get
\begin{equation}\label{cvs}
\begin{aligned}
    \la v, w-w_0\ra&=\la v_1, w-w_0\ra+\la v_2, w-w_0\ra \vspace{1ex}\\
    &\leq h(w)-h(w_0)
    \end{aligned}
\end{equation}
because $\la v_2, w-w_0\ra\leq 0$. Since \eqref{cvs} also holds if $w\notin \dom(h)$, i.e. $h(w)=\infty$, we see that $v\in \partial h(w_0)$, which completes the proof. 
\end{proof}

The theorem below provides an enhanced version of the method of Lagrange multipliers for convex optimization problems. This version applies to cases where the objective function can be represented as the sum of a continuous convex function and a subdifferentiable convex function. 

\begin{Theorem}\label{KKT} 
Consider the convex optimization problem \eqref{eq:cvx_opt} in a Hilbert space $\mathcal{H}$.
Let $w_0$ be a feasible solution to this problem such that $\partial h(w_0)\neq\emptyset$. Suppose that the Slater condition is satisfied for this problem, i.e., there exists $u\in \mathcal{H}$ such that 
\begin{enumerate}
   \item $u\in \dom(h)$,
   \item $g_i(u)<0$ for $i=1, \ldots, m$.
\end{enumerate}
Then, $w_0$ is an optimal solution to \eqref{eq:cvx_opt} if and only if there exist Lagrange multipliers $\lambda_i\geq 0$ for $i=1, \ldots, m$ such that
\begin{equation}\label{KKTC}
    0\in \partial f(w_0)+\partial h(w_0)+\sum_{i=1}^m \lambda_i\partial g_i(w_0),
\end{equation}
and $\lambda_i g_i(w_0)=0$ for $i=1, \ldots, m$.
\end{Theorem}
\begin{proof} Fix any feasible solution $w_0$ to \eqref{eq:cvx_opt} such that $\partial h(w_0)\neq\emptyset$. \\[1ex]
$\Longrightarrow$: Suppose that $w_0$ is an optimal solution to \eqref{eq:cvx_opt} under the Slater condition. We consider two cases.\\[1ex]
{\bf Case 1}: $g_i(w_0)<0$ for $i=1, \ldots, m$. Define the constraint set
\begin{equation*}
    \Omega=\big\{w\in \mathcal H\; \big |\; g_i(w)\leq 0\ \; \mbox{\rm for all }i=1, \ldots, m\big\}.
\end{equation*}
Obviously, we have $w_0\in \mbox{\rm int}(\Omega)$. 
Since $w_0$ is optimal to \eqref{eq:cvx_opt}, by both Fermat's rule and the subdifferential sum rule (see, e.g., \cite[Theorem~3.48]{bmn2022}), we can show that
\begin{equation*}
    0\in \partial (f+h)(w_0)+N(w_0; \Omega)=\partial f(w_0)+\partial h(w_0).
\end{equation*}
Consequently, \eqref{KKTC} holds with $\lambda_i=0$ for $i=1, \ldots, m$. \\[1ex]
{\bf Case 2}: $g_i(w_0)=0$ for some $i=1, \ldots, m$. 
Since $w_0$ is optimal to \eqref{eq:cvx_opt}, it is straightforward to check that $w_0$ is also an optimal solution to  the problem
\begin{equation*}
    \mbox{\rm min}\; \ph(w)=\max_{i=1, \ldots, m}\{f(w)+h(w)-f(w_0)-h(w_0), g_i(w)\}, \; w\in \mathcal H.
\end{equation*}
Let $g=\max_{i=1, \ldots, m}g_i$. Since $\dom(f+h)=\dom(h)$ due to the fact that $\dom(f)=\mathcal H$, by Theorem \ref{MaxRule}, we have
\begin{equation}\label{subin}
    0\in \partial \ph(w_0)\subset \mbox{\rm co}\big(\partial (f+h)(w_0)\cup \partial g(w_0)\big)+N(w_0; \dom(h)).
    \end{equation}
Consider the  active index set
\begin{equation*}
    I(w_0)=\big\{i=1, \ldots, m\; \big |\; g_i(w_0)=0\big\}\neq \emptyset.
\end{equation*}
In this case, we have $g(w_0)=0$, and by \cite[Theorem~3.59]{bmn2022} we conclude that
\begin{equation*}
    \partial g(w_0)=\co\big(\bigcup_{i\in I(w_0)}\partial g_i(w_0)\big).
\end{equation*}
From \eqref{subin}, there exist $\lambda_0\geq 0$ and $\gamma_0\geq 0$ along with $v_0\in \partial f(w_0)+\partial h(w_0)$, $\hat{v}\in \partial g(w_0)$ and $v\in N(w_0; \dom(h))$ such that $\lambda_0+\gamma_0=1$ and
\begin{equation*}
    0=\lambda_0 v_0+\gamma_0 \hat{v}+v. 
\end{equation*}
Since $\hat{v}\in \partial g(w_0)$, we have the representation
\begin{equation*}
    \hat{v}=\sum_{i\in I(w_0)}\gamma_i v_i,
\end{equation*}
where $\gamma_i\geq 0$, $v_i\in \partial g_i(w_0)$ for $i\in I(w_0)$ and $\sum_{i\in I(w_0)}\gamma_i=1$. 

Now, set $\lambda_i=0$ if $i\in \{1, \ldots, m\}\setminus I(w_0)$, and $\lambda_i=\gamma_0 \gamma_i$ for $i\in I(w_0)$. 
Then, we can show that $\sum_{i=0}^m \lambda_i = \lambda_0 + \gamma_0\sum_{i \in I(w_0)}\gamma_i = \lambda_0 + \gamma_0 = 1$. 
Moreover, we also have the representation
\begin{equation}\label{zero sum}
    0=\sum_{i=0}^m \lambda_i v_i +v,
\end{equation}
where $\lambda_i g_i(w_0)=0$ for $i=1, \ldots, m$. 

Next, we will show that $\lambda_0>0$ under the Slater condition. Suppose on the contrary that $\lambda_0=0$. Then, we have  $\sum_{i=1}^m\lambda_i=1$ and 
\begin{equation*}
    0=\sum_{i=1}^m \lambda_i v_i +v.
\end{equation*}
Let $u$ be an element in $\mathcal H$ that satisfies the Slater condition. Since $u\in \dom(h)$, by the definition of convex normal cone, we have $\la v, u-w_0\ra\leq 0$. Using the last relation, we have
\begin{equation*}
    0=\sum_{i=1}^m \la \lambda_i v_i,  u-w_0\ra+\la v,  u-w_0\ra\leq \sum_{i=1}^m \lambda_i(g_i( u)-g_i(w_0))=\sum_{i=1}^m \lambda_i g_i(u)<0,
\end{equation*}
which yields a contradiction, and thus we conclude that $\lambda_0>0$. 
By dividing by $\lambda_0$, redefining $\lambda_i$ as $\lambda_i \leftarrow \lambda_i/\lambda_0$, and using the cone property of $N(w_0; \dom(h))$ in \eqref{zero sum}, we can show that
\begin{equation*}
    0\in \partial f(w_0)+\partial h(w_0)+\sum_{i=1}^m \lambda_i \partial g_i(w_0)+N(w_0; \dom(h))=\partial f(w_0)+\partial h(w_0)+\sum_{i=1}^m \lambda_i\partial g_i(w_0),
\end{equation*}
where the last equality holds by Lemma \ref{subE}. Therefore, \eqref{KKTC} also holds with $\lambda_i g_i(w_0)=0$ for $i=1, \ldots, m$.

$\Longleftarrow$: Suppose that there exist $\lambda_i\geq 0$ such that \eqref{KKTC} holds with $\lambda_ig_i(w_0)=0$ for $i=1, \ldots, m$. 
Then, we have the representation
\begin{equation*}
    0=v_0+\sum_{i=1}^m \lambda_i v_i,
\end{equation*}
where $v_0\in \partial f(w_0)+\partial h(w_0)$, and $v_i\in \partial g_i(w_0)$ for $i=1, \ldots, m$. 
Taking any feasible solution $w$ of \eqref{eq:cvx_opt}, we have
\begin{eqnarray*}
\begin{array}{lcl}
    0 &= & \la v_0, w-w_0\ra+\sum_{i=1}^m \lambda_i\la v_i, w-w_0\ra \vspace{1ex}\\
    &  \leq & (f+h)(w)-(f+h)(w_0)+\sum_{i=1}^m\lambda_i(g_i(w)-g_i(w_0)) \vspace{1ex}\\
    &= & (f+h)(w)-(f+h)(w_0)+\sum_{i=1}^m \lambda_ig_i(w) \vspace{1ex}\\
    & \leq & (f+h)(w)-(f+h)(w_0).
    \end{array}
\end{eqnarray*}
This implies that $f(w_0)+h(w_0)\leq f(w)+h(w)$, and thus $w_0$ is an optimal solution to \eqref{eq:cvx_opt}. 
\end{proof}

\begin{Remark}
{\rm It follows from the proof of Theorem \ref{KKT} that the converse implication in the theorem does not require the Slater condition.
}
\end{Remark}

Next, we consider the optimization problem:
\begin{equation}\label{Q4}
\left\{\begin{array}{ll}
{\displaystyle\min_{ w \in \mathcal{H}}} &f(w), \vspace{1ex}\\
\mathrm{s.t.} &  g_i(w) \leq 0\; \mbox{\rm for }i=1, \ldots,m, \vspace{1ex}\\
&w\in \Theta,
\end{array}\right.
\end{equation}
where $f\colon \mathcal H\to \R$ and $g_i\colon \mathcal H\to \R$ are continuous convex functions, and $\Theta$ is a nonempty convex set in $\mathcal H$. 
\begin{Definition}\label{Slater Inequality} We say that the Slater condition is satisfied for \eqref{Q4} if there exists $u\in \mathcal H$ such that
\begin{enumerate}
   \item $u\in \Theta$,
   \item $g_i(u)<0$ for $i=1, \ldots, m$.
\end{enumerate}
    
\end{Definition}
Define the Lagrange function for \eqref{Q4}  by
\begin{equation*}
\mathcal{L}(w,\lambda) = f(w)+ \sum_{i=1}^m \lambda_i g_i(w),
\end{equation*}
where $\lambda=(\lambda_1, \ldots, \lambda_m)\in \R^m$ is the vector of Lagrange multipliers. 

The corollary below follows directly from Theorem \ref{KKT}.

\begin{Corollary}\label{KKT_Inequality} Consider the convex optimization problem  \eqref{Q4} and assume that the Slater condition in Definition \ref{Slater Inequality} is satisfied for this problem.  Let $w_0$ be a feasible solution to \eqref{Q4}. 
Then, $w_0$ is an optimal solution to this optimization problem if and only if there exists $\lambda=(\lambda_1,\ldots,\lambda_m)\in\R_+^m$ such that
\begin{equation}\label{kkt4}
0\in \partial_w\mathcal{L}(w_0, \lambda)+N(w_0; \Theta)= \partial f(w_0)+\sum_{i=1}^m\lambda_i\partial g_i(w_0)+N(w_0; \Theta),
\end{equation}
with $\lambda_i g_i(w_0)=0$ for all $i=1,\ldots,m$.
  
\end{Corollary}
\begin{proof} Let $w_0$ be a feasible solution to \eqref{Q4}. Define the function $h\colon \mathcal H\to \oR$ by
\begin{equation*}
    h(w)=\delta(w; \Theta), \ \; w\in \mathcal H.
\end{equation*}
Then, the optimization \eqref{Q4} reduces to \eqref{eq:cvx_opt}. We have $\partial h(w_0)=N(w_0; \Theta)\neq\emptyset$ because always have $0\in N(w_0; \Theta)$. By Theorem \ref{KKT}, the element $w_0$ is an optimal solution to \eqref{Q4} if and only if there exist $\lambda_i\geq 0$ for $i=1, \ldots, m$ such that \eqref{kkt4} holds with $\lambda_i g_i(w_0)=0$ for $i=1, \ldots, m$.  This completes the proof.     
\end{proof}

Next, we further discuss another consequence of the general result from Theorem \ref{KKT} for convex optimization problems with affine constraints using the notion of quasi-relative interior. Consider the following optimization problem:
\begin{equation}\label{Q3}
\left\{\begin{array}{ll}
{\displaystyle\min_{ w \in \mathcal{H}}} &f(w), \vspace{1ex}\\
\mathrm{s.t.} &  g_i(w) \leq 0\; \mbox{\rm for all }i=1, \ldots,m, \vspace{1ex}\\
&\la v_j, w\ra\leq b_j\; \mbox{\rm for }j=1, \ldots, p, \vspace{1ex}\\
&w\in \Theta,
\end{array}\right.
\end{equation}
where $f, g_i\colon \mathcal H\to \R$ for $i=1, \ldots, m$ are continuous convex functions,  $v_j\in \mathcal H$ and $b_j\in \R$ for $j=1, \ldots, p$, and $\Theta$ is a nonempty convex set in $\mathcal H$. 

\begin{Definition}\label{MixCon}
We say that the Slater condition holds for \eqref{Q3} if there exists $u\in \mathcal H$ such that:
 \begin{enumerate}
 \item $u\in \qri(\Theta)$,
  \item $g_i(u)<0$ for $i=1, \ldots, m$,
 \item $\la v_j, u\ra\leq b_j$ for $j=1, \ldots, p$.
  \end{enumerate}
  \end{Definition}

 Define the  Lagrange function for \eqref{Q3}  by
\begin{equation*}
\mathcal{L}(w,\lambda, \mu) = f(w)+ \sum_{i=1}^m \lambda_i g_i(w)+\sum_{j=1}^p \mu_j h_j(w),
\end{equation*}
where $\lambda=(\lambda_1, \ldots, \lambda_m)\in \R^m$, $\mu = (\mu_1, \cdots, \mu_p) \in \R^p$, and $h_j(w)=\la v_j, w\ra-b_j$ for $w\in \mathcal H$.

\begin{Theorem}\label{kttgpc}   
Consider the convex optimization problem  \eqref{Q3} and assume that the Slater condition in Definition \ref{MixCon} is satisfied for this problem.  Let $w_0$ be a feasible solution to \eqref{Q3}. 
Then, $w_0$ is an optimal solution to this problem if and only if there exist  $\lambda=(\lambda_1,\ldots,\lambda_m)\in\R_+^m$ and  $\mu=(\mu_1, \ldots, \mu_p)\in \R^p_+$ such that
\begin{equation}\label{kkt3}
0\in \partial_w\mathcal{L}(w_0, \lambda, \mu)+N(w_0; \Theta)= \partial f(w_0)+\sum_{i=1}^m\lambda_i\partial g_i(w_0)+\sum_{j=1}^p \mu_j v_j+N(w_0; \Theta),
\end{equation}
with $\lambda_i g_i(w_0)=0$ for $i=1,\ldots,m$ and $\mu_j h_j(w_0) = 0$ for $j=1, \ldots, p$.
\end{Theorem}
\begin{proof} Let $w_0$ be a feasible solution to \eqref{Q3}. Consider the polyhedral convex set \begin{equation}\label{Poly1}
P=\big\{w\in \mathcal H\; \big |\; \;  h_j(w)\leq 0\; \mbox{\rm for }j=1, \ldots, p\big\},
\end{equation}
and define the function $h\colon \mathcal H\to \oR$ by
\begin{equation*}
    h(w)=\delta(w; \Theta\cap P), \ \; w\in \mathcal H.
\end{equation*}
Then, the optimization \eqref{Q3} reduces to \eqref{eq:cvx_opt}. We have $\partial h(w_0)=N(w_0; \Theta\cap P)\neq\emptyset$. By Theorem \ref{KKT}, the element $w_0$ is an optimal solution to \eqref{Q3} if and only if there exist $\lambda_i\geq 0$ for $i=1, \ldots, m$ such that
\begin{equation*}
    0\in \partial f(w_0)+\sum_{i=1}^m \lambda_i\partial g_i(w_0)+N(w_0; \Theta\cap P),
\end{equation*}
where $\lambda_i g_i(w_0)=0$ for $i=1, \ldots, m$. 
Observe that $\mbox{\rm qri}(\Theta)\cap P\neq\emptyset$ under the Slater condition. 
Hence, by Theorem \ref{PNIR}, we have
\begin{equation*}
    N(w_0; \Theta\cap P)=N(w_0; \Theta)+N(w_0; P).
\end{equation*}
Using basic properties of linear algebra, we have
  \begin{equation*}
    N(w_0; P)={\rm cone}\big\{v_j\;\big|\;j\in J(w_0)\big\}=\big\{\sum_{j\in J(w_0)}\mu_j v_j\; \big |\; \mu_j\geq 0\; \mbox{\rm for }j\in J(w_0)\big\},
\end{equation*}
where $J(w_0)=\{j=1, \ldots, p\; |\; h_j(w_0)=0\}$.
Therefore, we can easily show that $w_0$ is an optimal solution to \eqref{Q3} iff \eqref{KKT2} holds with $\lambda_i g_i(w_0)=0$ for $i=1, \ldots, m$ and and $\mu_j h_j(w) = 0$ for $j=1, \ldots, p$. This completes the proof. 
\end{proof}
\begin{Remark}\label{KKTRn} {\rm 
Following the proof of Theorem \ref{kttgpc}, we see that this theorem is valid under  weaker Slater conditions (a) and (b) below:
\begin{enumerate}
\item There exists $u\in \Theta\cap P$ such that
\begin{equation*}
    g_i(u)<0\; \ \mbox{\rm for }i=1, \ldots, m.
\end{equation*}
    \item There exists $\tilde u\in \qri(\Theta)$ such that
    \begin{equation*}
        h_j(\tilde u)\leq 0\; \; \mbox{\rm for }j=1, \ldots, p.
    \end{equation*}
\end{enumerate}}
\end{Remark}

To conclude this subsection, we discuss another consequence of the general result from Theorem \ref{KKT}.

\begin{Definition}\label{SlaterCQ2}
Consider the problem \eqref{Q2}.  We say that the Slater condition is satisfied for  \eqref{Q2} if there exists $u\in \mathcal H$ such that:
 \begin{enumerate}
 \item $u\in \qri(\Theta)$,
 \item $g_i(u)<0$ for all $i=1, \ldots, m$,
 \item $\mathcal A u=b$.
 \end{enumerate}
 \end{Definition}
 
Define the  Lagrange function for \eqref{Q2}  by
\begin{equation}\label{LF1}
\mathcal{L}(w, \lambda, \eta)=f(w)+\sum_{i=1}^m \lambda_i g_i(w)+\la \eta, \mathcal Aw-b\ra,\; w\in \mathcal H,
\end{equation}
where $\lambda=(\lambda_1, \ldots, \lambda_m)\in \R^m$ and $\eta \in \R^q$.

The theorem below can be proved by repeating the proof of Theorem \ref{kttgpc} using the polyhedral convex set
$$P=\big\{w\in \mathcal{H}\;\big |\; \mathcal{A}w=b\big\}$$
instead of the set $P$ defined in \eqref{Poly1}. For this new set $P$, by basic properties of linear algebra we observe that
\begin{equation*}
    N(w_0; P)=(\mbox{\rm ker}(\mathcal{A}))^\perp=\mbox{\rm im}(\mathcal{A}^*).
\end{equation*}
In fact, this theorem can also be derived from Theorem \ref{kttgpc}.
\begin{Theorem} \label{KKT_2} 
Consider the convex optimization problem \eqref{Q2} in a Hilbert space $\mathcal{H}$.
Suppose that the Slater condition in Definition \ref{SlaterCQ2} is satisfied for \eqref{Q2} and $w_0$ is a feasible solution to this problem. Then, $w_0$ is an optimal solution to \eqref{Q2} if and only if there exist  $\lambda=(\lambda_1,\ldots,\lambda_m)\in\R_+^m$ and $\eta\in \R^q$ such that
\begin{equation}\label{KKT2}
0\in \partial_w\mathcal{L}(w_0, \lambda, \eta)+N(w_0; \Theta)=\partial f(w_0)+\sum_{i=1}^m\lambda_i\partial g_i(w_0)+\mathcal{A}^*\eta+N(w_0; \Theta),
\end{equation}
with $\lambda_i g_i(w_0)=0$ for $i=1,\ldots,m$.
Here, $\mathcal{A}^{*}$ is the adjoint mapping of $\mathcal{A}$.
\end{Theorem}

\begin{Remark}\label{KKTR} {\rm 
Similar to Remark \ref{KKTRn}, the conclusion of Theorem \ref{KKT_2}  remains valid under weaker  Slater conditions:
\begin{enumerate}
\item There exists $u\in \Theta\cap P$ such that
\begin{equation*}
    g_i(u)<0\; \ \mbox{\rm for }i=1, \ldots, m.
\end{equation*}
    \item There exists $\tilde u\in \qri(\Theta)$ such that
    \begin{equation*}
       \mathcal A(\tilde u)=b.
    \end{equation*}
\end{enumerate}
}
\end{Remark}

\subsection{Lagrangian  Duality in Constrained Convex Optimization}\label{sec:Lag_duality_for_cvx}
This subsection presents a unified weak and strong duality theory for the constrained convex optimization problem \eqref{Q2}. 

Using the Lagrange function \eqref{LF1}, we define the dual function $\Hat{\mathcal{L}}\colon \R^m_+\times \R^q\to [-\infty, \infty)$ by
\begin{equation*}
\Hat{\mathcal{L}}(\lambda,\eta)=\inf \big\{\mathcal{L}(w, \lambda, \eta)\; \big |\; w\in \Theta\big\}.
\end{equation*}
\begin{Proposition} The function $\Hat{\mathcal{L}}\colon \R^m_+\times \R^q\to [-\infty, \infty)$ is a concave function.
\end{Proposition}
\begin{proof} By definition, $\Hat{\mathcal{L}}$ is concave, as it is the infimum of a collection of affine (and therefore concave) functions defined on $\mathbb{R}^m_+ \times \mathbb{R}^q$.
\end{proof}

Now, we consider the maximization problem:
\begin{equation}\label{LD}
\left\{\begin{array}{ll}
{\displaystyle\max}& \Hat{\mathcal{L}}(\lambda, \eta), \vspace{1ex}\\
\mathrm{s.t.} &  \lambda\in \R^m_+, \;\eta\in \R^q.
\end{array}\right.
\end{equation}
This problem is referred to as the Lagrange dual problem of the primal problem \eqref{Q2}.

To continue, we define the constraints of the primal problem and the dual problem, respectively, by
\begin{eqnarray*}
\begin{array}{ll}
&\Omega_p=\big\{w\in \mathcal H\; \big|\; g_i(w)\leq 0\; \mbox{\rm for }i=1, \ldots, m,\; \mathcal A(w)=b\big\}\cap \Theta, \vspace{1ex}\\
&\Omega_d=\R^m_+\times \R^q.
\end{array}
\end{eqnarray*}

\begin{Theorem}\label{LWD} 
Define the optimal value of the primal problem \eqref{Q2} and the optimal value of the dual problem \eqref{LD}, respectively, as follows:
\begin{eqnarray*}
\begin{array}{ll}
&p=\inf\{f(w)\; |\; w\in \Omega_p\}, \vspace{1ex}\\
&d=\sup\{\Hat{\mathcal{L}}(\lambda, \eta)\; |\; (\lambda,\eta)\in \Omega_d\}.
\end{array}
\end{eqnarray*}
Then, we always have
\begin{equation}\label{supinf}
\inf_{w\in \Theta} \sup_{(\lambda,\eta)\in \Omega_d}\mathcal{L}(w, \lambda, \eta)=p\quad \mbox{\rm and } \sup_{(\lambda, \eta)\in \Omega_d}\inf_{w\in \Theta} \mathcal{L}(w, \lambda, \eta)=d.
\end{equation}
Moreover, weak duality $d\leq p$ holds.
\end{Theorem}
\begin{proof} Observe that for any $w\in \Theta$, which contains $\Omega_p$, we have
\begin{equation*}
\sup_{(\lambda,\eta)\in \Omega_d}\mathcal{L}(w, \lambda, \eta)=
\begin{cases}
f(w)&\; \mbox{\rm if }\;w\in \Omega_p,\\
\infty &\; \mbox{\rm if }\;w\notin \Omega_p.
\end{cases}
\end{equation*}
It follows that
\begin{equation*}
\inf_{w\in \Theta} \sup_{(\lambda,\eta)\in \Omega_d}\mathcal{L}(w, \lambda, \eta)=\inf_{w\in \Omega_p}f(w)=p.
\end{equation*}
By the definition of the dual function and $d$, we get
\begin{equation*}
\sup_{(\lambda, \eta)\in \Omega_d}\inf_{w\in \Theta} \mathcal{L}(w, \lambda, \eta)=\sup_{(\lambda, \eta)\in \Omega_d}\Hat{\mathcal{L}}(\lambda,\eta)=d.
\end{equation*}
The last conclusion now follows from the well-known minimax theorem; see, e.g., \cite{bmn24}. 
\end{proof}

\begin{Definition} An element $(\bar w, \bar \lambda, \bar \eta)\in \Theta\times \Omega_d$ is called a saddle point of the Lagrange function $\mathcal{L}$ defined in \eqref{LF1}  if
\begin{equation}\label{eq:saddle_point}
\mathcal{L}(\bar w, \lambda, \eta)\leq \mathcal{L}(\bar w, \bar \lambda, \bar \eta)\leq \mathcal{L}(w, \bar \lambda, \bar \eta) \quad \mbox{\rm for all }w\in \Theta, (\lambda, \eta)\in \Omega_d.
\end{equation}
\end{Definition}
Now, we state conditions for attaining strong duality in the following theorem.

\begin{Theorem}\label{lsd1} Consider the primal problem \eqref{Q2}  and its Lagrange dual problem \eqref{LD}. \\[1ex]
{\rm(a)} Suppose that the KKT conditions  hold for $\bar w\in \Omega_p$ and $(\bar{\lambda}, \bar{\eta})\in \Omega_d$, i.e. we have
\begin{equation}\label{ldsub}
0\in \partial_w\mathcal{L}(\bar w, \bar \lambda, \bar{\eta})+N(\bar w; \Theta)=\partial f(\bar w)+\sum_{i=1}^m \bar{\lambda}_i \partial g_i(\bar w)+\mathcal A^*\bar{\eta}+N(\bar w; \Theta)
\end{equation}
and complementary slackness $\bar{\lambda}_i g_i(\bar w)=0$ holds for all $i=1,\ldots,m$. 
Then, we obtain Lagrangian strong duality
\begin{equation*}
p=d. 
\end{equation*}
Furthermore, $\bar w$ is an optimal solution to the primal problem \eqref{Q2}, and $(\bar\lambda, \bar \eta)$ is an optimal solution to the dual problem \eqref{LD}. \\[1ex]
{\rm (b)} Under Lagrangian strong duality $p=d$, if $\bar w$ is an optimal solution to the primal problem \eqref{Q2}  and $(\bar{\lambda},\bar\eta)$ is an optimal solution to the dual problem \eqref{LD}, then the KKT conditions stated in {\rm(a)} are satisfied.  Furthermore, $(\bar w, \bar \lambda, \bar \eta)$ is a saddle point of $\mathcal{L}$.
\end{Theorem}

\begin{proof} 
(a) Suppose that the KKT conditions hold for $\bar w\in \Omega_p$ and $(\bar \lambda, \bar\eta)\in \Omega_d$. 
First, by \eqref{ldsub}, we have
\begin{equation}\label{sdpr}
    \mathcal{L}(\bar w, \bar \lambda,\bar\eta)\leq \mathcal{L}(w, \bar \lambda, \bar\eta)\ \; \mbox{\rm for all }w\in \Theta.
\end{equation}
This proves the second inequality in \eqref{eq:saddle_point}.

Next, under the complementary slackness condition we also have $\mathcal{L}(\bar w, \bar \lambda, \bar\eta)=f(\bar w)$. 
Using this fact, $\lambda_ig_i(\bar w) \leq 0$ for all $i=1,\ldots, m$ and $\lambda\in\R^m$, and $\mathcal{A}\bar{w} - b = 0$, we can show that
\begin{equation}\label{sdpl}
    \mathcal{L}(\bar w, \lambda, \eta)=f(\bar w)+\sum_{i=1}^m \lambda_i g_i(\bar w)+\la \eta, A\bar w-b\ra\leq f(\bar w)=\mathcal{L}(\bar w, \bar \lambda, \bar \eta)\ \; \mbox{\rm for all }(\lambda, \eta)\in \Omega_d.
\end{equation}
This proves the first inequality of \eqref{eq:saddle_point}.
Thus $(\bar w, \bar \lambda, \bar \eta)$ is a saddle point of $\mathcal{L}$.
Consequently, $p=d$ by \eqref{supinf} from Theorem~\ref{LWD}; see  \cite{bmn24}. 

Using the KKT conditions and the converse implication of Theorem~\ref{KKT}, which does not require the Slater condition, we see that $\bar w$ is an optimal solution to the primal problem \eqref{Q2}. It follows from \eqref{sdpr}
that
\begin{equation*}
   \mathcal{L}(\bar w, \bar \lambda, \bar \eta)\leq \inf_{w\in \Theta}\mathcal{L}(w, \bar \lambda, \bar \eta)=\Hat{\mathcal{L}}(\bar \lambda, \bar \eta).
\end{equation*}
Hence, from \eqref{sdpl}, we can see that
\begin{equation*}
    \Hat{\mathcal L}(\lambda, \eta)\leq \mathcal{L}(\bar w, \lambda, \eta)\leq \mathcal{L}(\bar w, \bar \lambda, \bar \eta)\leq \Hat{\mathcal{L}}(\bar \lambda, \bar \eta)\; \mbox{\rm for all }(\lambda, \eta)\in \Omega_d,
\end{equation*}
which means that $(\bar \lambda, \bar \eta)$ is an optimal solution to the dual problem \eqref{LD}. 

(b) Under Lagrangian strong duality $p=d$, suppose that $\bar w$ is an optimal solution to the primal problem \eqref{Q2} and $(\bar{\lambda}, \bar \eta)$ is an optimal solution to the dual problem \eqref{LD}. Then, we have
\begin{equation*}
\begin{aligned}
    f(\bar w)=p=d=\Hat{\mathcal{L}}(\bar \lambda, \bar \eta)&=\inf_{w\in \Theta}\mathcal{L}(w, \bar \lambda, \eta)\leq \mathcal{L}(\bar w, \bar \lambda, \bar \eta)\\
    &=f(\bar w)+\sum_{i=1}^m \bar \lambda_i g_i(\bar w)+\la \bar \eta, A\bar w-b\ra\leq f(\bar w).
    \end{aligned}
\end{equation*}
Since $\mathcal{A}\bar{w} - b = 0$, the last inequality implies that $\sum_{i=1}^m \bar \lambda_i g_i(\bar w)=0$.
Since each term in this sum is non-positive, we obtain $\bar \lambda_i g_i(\bar w)=0$ for all $i=1, \ldots, m$. 
Consequently, we get
\begin{equation*}
    \inf_{w\in \Theta}\mathcal{L}(w, \bar \lambda, \bar \eta)=f(\bar w)=f(\bar w)+\sum_{i=1}^m \bar \lambda_i g_i(\bar w)=\mathcal{L}(\bar w, \bar \lambda, \bar \eta).
\end{equation*}
This condition means that $\mathcal{L}(\cdot, \bar \lambda, \bar \eta)$ has an absolute minimum on $\Theta$ at $\bar w$. 
By Fermat's rule, we have
\begin{equation*}
    0\in \partial_w\mathcal{L}(\bar w, \bar \lambda, \bar \eta)+N(\bar w; \Theta)=\partial f(\bar w)+\sum_{i=1}^m \bar \lambda_i\partial g_i(\bar w)+\mathcal A^*\bar \eta+N(\bar w; \Theta).
\end{equation*}
Therefore, the KKT conditions are satisfied. 
Finally, it is straightforward to prove that $(\bar w, \bar \lambda, \bar \eta)$ is a saddle point of the Lagrange function $\mathcal{L}$. 
\end{proof}

Our next goal is to establish a Lagrangian strong duality theorem for \eqref{Q2} under the Slater condition in Definition \ref{SlaterCQ2}.  To proceed, consider the following system with the unknown $w$:
\begin{eqnarray}\label{SP21}
\begin{cases}
f(w)<0, \\
g_i(w)\leq 0\; \ \mbox{\rm for all }i=1, \ldots, m, \\
\mathcal A w=b,\\
w\in \Theta.
\end{cases}
\end{eqnarray}
We also consider the following system  with the unknowns $\lambda=(\lambda_1, \ldots, \lambda_m)$ and $\eta$:
\begin{eqnarray}\label{SD21}
\begin{cases}
 f(w)+\sum_{i=1}^m\lambda_ig_i(w)+\la \eta, \mathcal Aw-b\ra\geq 0 \;  \ \mbox{\rm for all }w\in \Theta,\\
 \lambda_i\geq 0\; \mbox{\rm for all }i=1, \ldots, m, \;\; \eta\in \R^q.
\end{cases}
\end{eqnarray}

\begin{Lemma} \label{SPD2} Consider the systems \eqref{SP21} and \eqref{SD21}. 
Suppose that the Slater condition in Definition \ref{SlaterCQ2} is satisfied for \eqref{Q2}. 
Then, \eqref{SP21} has a solution if and only if \eqref{SD21} has no solution.
\end{Lemma}

\begin{proof} Suppose that \eqref{SP21} has a solution $w_0$.  Let us show that \eqref{SD21} has no solution. By contradiction, suppose that it has a solution $\lambda=(\lambda_1, \ldots, \lambda_m)\in \R^m_+$, $\eta\in \R^q$. 

On the one hand, from \eqref{SP21}, we have
\begin{equation*}
f(w_0)+\sum_{i=1}^m\lambda_ig_i(w_0)+\la \eta, Aw_0-b\ra\leq f(w_0)<0.
\end{equation*}
On the other, the first line of \eqref{SD21} shows that $f(w_0)+\sum_{i=1}^m\lambda_ig_i(w_0)+\la \eta, Aw_0-b\ra \geq 0$.
Both inequalities lead to a contradiction. Note that in this proof, the Slater condition is not required.

Now, suppose that \eqref{SD21} has no solution. Let us show that \eqref{SP21} has a solution. By contradiction,  suppose that \eqref{SP21} has no solution. Consider the set
\begin{equation*}
\Omega=\big\{(\gamma_0, \gamma, z)\in \R\times \R^m\times \R^q\; \big|\; f(w)< \gamma_0, g(w)\preceq  \gamma, \mathcal Aw-b= z\; \mbox{\rm for some }w\in \Theta\big\},
\end{equation*}
where $\gamma_0\in \R$, $\gamma=(\gamma_1, \ldots, \gamma_m)\in \R^m$, $z\in \R^q$, and $g(w)=(g_1(w), \ldots, g_m(w))$ for $w\in \mathcal H$. Here, the notation $g(w)\preceq  \gamma$ means that $g_i(w)\leq \gamma_i$ for all $i=1, \ldots, m$. 
Clearly, $\Omega$ is nonempty by the Slater condition, and  $(0, 0, 0)\notin\Omega$ because \eqref{SP21} has no solution. 
Moreover, the convexity of $\Omega$ follows from the convexity of $f,g_i$, and the linearity of $\mathcal A$.

By the convex proper separation theorem in the finite-dimensional space (see, e.g., \cite[Theorem 2.92]{bmn2022}), there exist $(\lambda_0, \lambda_1, \ldots, \lambda_m)\in \R^{m+1}_+$ and $\eta\in \R^q$ such that
\begin{equation}\label{dual separation2}
\lambda_0\gamma_0+\sum_{i=1}^m\lambda_i\gamma_i+\la \eta, z\ra\geq 0
\end{equation}
whenever $(\gamma_0, \gamma, z)\in \Omega$ with $\gamma=(\gamma_1, \ldots, \gamma_m)$. In addition, there exists $(\tilde{\gamma}_0, \tilde{\gamma}, \tilde{z})\in \Omega$ with $\tilde{\gamma}=(\tilde{\gamma}_1, \ldots, \tilde{\gamma}_m)$ such that
\begin{equation}\label{proper separation}
\lambda_0\tilde{\gamma}_0+\sum_{i=1}^m\lambda_i\tilde{\gamma}_i+\la \eta, \tilde{z}\ra > 0.
\end{equation}
By passing to the limit, observe  that \eqref{dual separation2} also holds whenever $(\gamma_0, \gamma, z)\in \Bar\Omega$. 
We can show that $\lambda_i\geq 0$ for all $i=0, 1, \ldots, m$. 
Indeed, fix any $\gamma_0>\max\{0,f(u)\}$, where $u$ is an element that satisfies the Slater condition, and observe that $(\gamma_0,0,\ldots,0)\in \Omega$. By  \eqref{dual separation2} we have $\lambda_0\gamma_0\geq 0$, and thus $\lambda_0\geq 0.$ To show $\lambda_1\geq 0,$ by a contradiction, suppose that $\lambda_1<0.$ If $\lambda_0=0$, then for any $\gamma_1>0>g_1(u),$ we can see that $(\gamma_0,\gamma_1,0,\ldots, 0)\in \Omega$ but \eqref{dual separation2} is not satisfied. 
If $\lambda_0>0,$  choosing $\gamma_1>-\frac{\lambda_0\gamma_0}{\lambda_1}>g_1(u),$ we have $(\gamma_0,\gamma_1,0,\ldots, 0)\in\Omega$ but  \eqref{dual separation2} is not satisfied. This contradiction shows that $\lambda_1\geq 0$.
By a similar way, we can show that $\lambda_i\geq0$ for  $i=2,\ldots,m .$ 

Now, taking any $w\in \Theta$, we see that $(f(w), g_1(w), \ldots, g_m(w), \mathcal Aw-b)\in \Bar{\Omega}$. 
Thus we have
\begin{equation}\label{csi1}
\lambda_0 f(w)+\sum_{i=1}^m \lambda_ig_i(w)+\la \eta,  \mathcal Ax-b\ra\geq 0\; \mbox{\rm for all }w\in \Theta.
\end{equation}
If $\lambda_0=0$, then we can use the element $u\in \mbox{\rm qri}(\Theta)$ such that $g_i(u)<0$ for all $i=1, \ldots, m$ and $\mathcal Au=b$ to get
\begin{equation*}
\sum_{i=1}^m \lambda_ig_i(u)+\la \eta, \mathcal Au-b\ra=\sum_{i=1}^m\lambda_ig_i(u)\geq 0.
\end{equation*}
This implies that $\lambda_i=0$ for all $i=1, \ldots, m$. 
Then, we get
\begin{equation}\label{sepc}
\la \eta, \mathcal Aw-b\ra\geq 0\; \mbox{\rm for all }w\in \Theta.
\end{equation}
Next, we will show that $\eta=0$. By \cite[Proposition 2.10]{bl-qri} we have
\begin{equation*}
    b=\mathcal Au\in \mathcal A(\mbox{\rm qri}\, \Theta)=\mbox{\rm qri}(\mathcal A(\Theta)).
\end{equation*}
This and \eqref{sepc} imply that
\begin{equation*}
    \la \eta, y\ra\geq 0\; \mbox{\rm for all }y \in \overline{\cone (\mathcal A(\Theta)-b)},
\end{equation*}
where the latter set is a linear subspace. 
Consequently, we obtain $\la \eta, y\ra=0$ for all $y\in \overline{\cone (\mathcal A(\Theta)-b)}$. By the definition of $\tilde{z},$ there exists $\tilde{w}\in\Theta\subset \mathcal H$ such that $\tilde{z}=\mathcal A\tilde{w}-b$. Then, we get
\begin{equation*}
\la \eta, \tilde{z}\ra=\la \eta , \mathcal A\tilde w-b\ra=0,
\end{equation*}
which yields a contradiction by \eqref{proper separation}.  Therefore, $\lambda_0>0$, and we can divide both sides of \eqref{csi1} to show that \eqref{SD21} has a solution, which is a contradiction. This completes the proof. 
\end{proof}

\begin{Theorem}\label{LSDE} 
Consider the convex optimization problem \eqref{Q2}.  
Suppose that the   Slater condition in Definition \ref{SlaterCQ2} is satisfied for this problem.  Then, Lagrangian strong duality holds, i.e., $p=d$.
\end{Theorem}

\begin{proof} 
If $p=-\infty$, then $d=-\infty$ by Theorem \ref{LWD}. Thus, we can assume that $p\in \R$. 
Then, the system
\begin{equation*}
\begin{cases}
f(w)-p<0,\\
g_i(w)\leq 0 \; \mbox{\rm for all }i=1, \ldots, m,\\
\mathcal Aw=b,\\
w\in \Theta
\end{cases}
\end{equation*}
has no solution. By Lemma \ref{SPD2},  there exist $\lambda=(\lambda_1, \ldots, \lambda_m)\in \R^m_{+}$ and $\eta\in \R^q$ such that
\begin{equation*}
f(w)-p+\sum_{i=1}^m \lambda_i g_i(w)+\la \eta, \mathcal Aw-b\ra\geq \; \mbox{\rm for all }w\in \Theta.
\end{equation*}
This implies that
\begin{equation*}
d\geq \Hat{\mathcal{L}}(\lambda, \eta)=\inf_{w\in \Theta}\big(f(w)+\sum_{i=1}^m \lambda_i g_i(w)+\la \eta, \mathcal Aw-b\ra\big) \geq p.
\end{equation*}
Combining this with Lagrangian weak duality from Theorem \ref{LWD} justifies the conclusion. 
\end{proof}

\begin{Remark}\label{KKTR1} {\rm 
Similar to Corollary \ref{KKT_Inequality}, if the equality constraint $\mathcal Aw=b$ is absent, then the conclusion of Theorem \ref{LSDE} holds under the Slater condition: there 
exists $u\in \mathcal H$ such that
 \begin{enumerate}
 \item $u\in \Theta$,
 \item $g_i(u)<0$ for all $i=1, \ldots, m$.
 \end{enumerate}
Furthermore, we can use a similar way to obtain a Lagrangian strong duality theorem for the optimization problem }\eqref{Q3} with convex inequality constraints and affine inequality constraints. 

\end{Remark}

\section{The Support Vector Machine Problem and Its Extensions}\label{svma}

\setcounter{equation}{0}

In this section, we utilize the results presented in the previous sections to study extensions of the classical SVM model as a case study. Given that our theory applies to Hilbert spaces, it enables us to analyze SVM models employing kernel representations in a Hilbert space.

\subsection{Hard-margin constrained SVM via Lagrangian duality}\label{svma hard margin}
To model the constrained SVM, we first present the well-known formulas for the distance from a point to a hyperplane and the metric projection of the point onto the hyperplane in a Hilbert space. These formulas can be derived using the method of Lagrange multipliers or Lagrangian duality.

\begin{Lemma}\label{dist hp} Let $\mathcal H$ be a  Hilbert space, and let $A$ be a hyperplane in $\mathcal H$ given by
\begin{equation*}
    A=\big\{x\in \mathcal H\; \big |\; \la w, x\ra+b=0\big\},
\end{equation*}
where $w\in \mathcal H\setminus\{0\}$ and $b\in \R.$ Then for any $x_0\in \mathcal H$ we have
\begin{equation*}
    d(x_0; A)=\frac{|\la w, x_0\ra+b|}{\|w\|}\;\ \mbox{\rm and }\ \mathcal P(x_0; A)=\big\{x_0-\left(\frac{\la w, x_0\ra+b}{\|w\|}\right)\frac{w}{\|w\|}\big\}.
\end{equation*}
\end{Lemma}

Next, we present another formula for computing the distance between two parallel hyperplanes in a Hilbert space. 

Recall that the distance between two sets $\Omega_1$ and $\Omega_2$ in a Hilbert space $\mathcal H$ is defined by
\begin{equation*}
    d(\Omega_1, \Omega_2)=\inf_{x_1, \in \Omega_1, \, x_2\in \Omega_2}\|x_1-x_2\|.
\end{equation*}

\begin{Corollary}\label{twohp} Let $A_1$ and $A_2$ be two hyperplanes in a Hilbert space $\mathcal H$ given by
\begin{equation*}
    A_1=\big\{x\in \mathcal H\; \big |\; \la w, x\ra+b_1=0\big\}\ \; \mbox{\rm and } \ A_2=\big\{x\in \mathcal H\; \big |\; \la w, x\ra+b_2=0\big\},
\end{equation*}
where $w\in \mathcal{H}\setminus\{0\}$ and $b_1, b_2\in \R$. Then we have
\begin{equation*}
    d(A_1, A_2)=\frac{|b_1-b_2|}{\|w\|}.
\end{equation*}
\end{Corollary}
\begin{proof}
Take any $x_0\in A_1$. 
Then,  $\la w, x_0\ra+b_1=0$, leading to $\la w, x_0\ra=-b_1$. By Lemma~\ref{dist hp} we have
\begin{equation*}
    d(x_0; A_2)=\frac{|\la w, x_0\ra+b_2|}{\|w\|}=\frac{|-b_1+b_2|}{\|w\|}=\frac{|b_1-b_2|}{\|w\|}.
\end{equation*}
 It follows that
\begin{equation*}
    d(A_1, A_2)=\inf_{x_1\in A_1, x_2\in A_2}\|x_1-x_2\|\leq d(x_0; A_2)=\frac{|b_1-b_2|}{\|w\|}.
\end{equation*}
Now, for any $x_1\in A_1$ and $x_2\in A_2$ we have
\begin{equation*}
    \la w, x_1\ra+b_1=0\; \ \mbox{\rm and }\ \la w, x_2\ra+b_2=0.
\end{equation*}
It follows that
\begin{equation*}
    |b_1-b_2|=|\la w, x_1\ra-\la w, x_2\ra|=|\la w, x_1-x_2\ra|\leq \|w\|\,\|x_1-x_2\|,
\end{equation*}
which yields the upper estimate
 \begin{equation*}
     \frac{|b_1-b_2|}{\|w\|}\leq \|x_1-x_2\|.
 \end{equation*}
Taking the infimum with respect to $x_1\in A_1$ and $x_2\in A_2$ gives us
\begin{equation*}
     \frac{|b_1-b_2|}{\|w\|}\leq \inf_{x_1\in A_1, x_2\in A_2}\|x_1-x_2\|=d(A_1, A_2).
\end{equation*}
Combining both inequalities completes our proof. 
\end{proof}

Let $\{\mathbf x_i\}_{i=1}^m$ be elements of $\mathcal H$ that belong to two groups. Here, we can treat $\mathbf{x}_i$ as the kernel representation of a data point $x_i$ through a feature mapping $\varphi$ as $\mathbf{x}_i = \phi(x_i)$, which often lies in a Hilbert space.
If $\mathbf x_i$ belongs to the first group, we label it by setting $y_i=1$, and if $\mathbf x_i$ belongs to the second group, we label it by setting $y_i=-1$. In this way, we form a training set $S=\{(\mathbf x_i, y_i)\}_{i=1}^m$, where $\mathbf x_i\in \mathcal H$ and $y_i\in\{-1, 1\}$. 
The classical SVM task seeks a hyperplane $\la w, x\ra+b=0$ with the \emph{largest margin} that separates two groups.  The margin is determined by two parallel hyperplanes with equal distance from the hyperplane $\la w, x\ra+b=0$ given by $\la w, x\ra+b=c$ and $\la w, x\ra+b=-c$, where $c\neq 0$.
Dividing each side of these equations by $c$, we can always assume that $c=1$. By Corollary~\ref{twohp}, the margin is computed as half of the distance between these two hyperplanes and is given by $m=\frac{1}{2}d=\frac{1}{\|w\|}$.

In order to satisfy the separation requirement, one can impose  the following conditions:
\begin{equation*}
y_i(\la \mathbf x_i, w\ra +b)\geq 1\; \mbox{\rm for all }i=1,\ldots,m.
\end{equation*}
The SVM task is now formulated as the following constrained maximization problem:
\begin{equation*}
\left\{\begin{array}{ll}
{\displaystyle\max_{ w \in \mathcal{H},\, b\in \R}} & \frac{1}{\|w\|}, \vspace{1ex}\\
\mathrm{s.t.} &  y_i(\la \mathbf x_i,w\ra+b)\geq 1\; \mbox{\rm for all }i=1,\ldots,m,
\end{array}\right.
\end{equation*}
which is equivalent to the minimization problem:
\begin{equation}\label{SVM0}
\left\{\begin{array}{ll}
{\displaystyle\min_{ w \in \mathcal{H},\, b\in \R}} & \frac{1}{2}\|w\|^2, \vspace{1ex}\\
\mathrm{s.t.} &  y_i(\la \mathbf x_i,w\ra+b)\geq 1\; \mbox{\rm for all }i=1,\ldots,m.
\end{array}\right.
\end{equation}
Note that classical SVM models do not require any constraint on $w$, making them easier to develop numerical methods for solving them.

To study the SVM problem \eqref{SVM0}, we often combine the bias term with $w$ to obtain a strongly convex problem.
The follow lemma shows that, by doing so, we obtain an approximate model up to any accuracy.

\begin{Lemma}\label{SVM0_approx}
Given any $\gamma > 0$, consider the following problem:
\begin{equation}\label{SVM0a}
\left\{\begin{array}{ll}
{\displaystyle\min_{ \hat{w} \in \hat{\mathcal{H}} }} & \frac{1}{2}\|\hat{w}\|^2, \vspace{1ex}\\
\mathrm{s.t.} &  y_i \la \hat{\mathbf x}_i, \hat{w} \ra \geq 1\; \mbox{\rm for all }i=1,\ldots,m,
\end{array}\right.
\end{equation}
where $\hat{w} = (w, z)$, $\hat{\mathcal{H}} = \mathcal{H}\times \R$, and $\hat{\mathbf{x}}_i = \big[\mathbf{x}_i, \frac{1}{\gamma}\big]$ for all $i=1,\ldots, m$.
Then, $(w, b)$ is a feasible solution to \eqref{SVM0} iff $\hat{w} = (w, \gamma b)$ is a feasible solution to \eqref{SVM0a}.
In addition, if the bias term $b$ in \eqref{SVM0} is in a compact interval $l_b \leq b \leq u_b$, then we have 
\begin{equation*}
\frac{1}{2}\Vert w \Vert^2 \leq \frac{1}{2}\Vert \hat{w}\Vert^2 \leq \frac{1}{2}\Vert w \Vert^2 + \frac{\gamma^2}{2}\max\{l_b^2, u_b^2\}.
\end{equation*}
\end{Lemma}

\begin{proof}
It is obvious to write $\la \mathbf{x}_i, w \ra + b = \la \mathbf{x}_i, w \ra + \frac{1}{\gamma} (\gamma b) = \la \big( \mathbf{x}_i, \frac{1}{\gamma}\big), (w, \gamma b) \ra = \la \hat{\mathbf{x}}_i, \hat{w} \ra$.
Therefore, $(w, b)$ satisfies the constraint $y_i(\la \mathbf{x}_i, w \ra + b) \geq 1$ iff $\hat{w} = (w, \gamma b)$ satisfies the constraint $y_i\la \hat{\mathbf{x}}_i, \hat{w} \ra \geq 1$.

In addition, we have $\Vert \hat{w} \Vert^2 = \Vert w \Vert^2 + \gamma^2b^2$. Thus, if  $b \in [l_b, u_b]$, then 
\begin{equation*}
\frac{1}{2}\Vert w \Vert^2 \leq \frac{1}{2}\Vert \hat{w} \Vert^2 = \frac{1}{2}\Vert w \Vert^2 + \frac{\gamma^2b^2}{2} \leq \frac{\gamma^2}{2}\max\{l_b^2, u_b^2\},
\end{equation*}
which completes the proof. 
\end{proof}

In this section, we consider the {\em constrained SVM} given by

\begin{equation}\label{SVMP1}
\left\{\begin{array}{ll}
{\displaystyle\min_{ w \in \mathcal{H}}} & \frac{1}{2}\|w\|^2, \vspace{1ex}\\
\mathrm{s.t.} & y_i(\la \mathbf x_i,w\ra)\geq 1\; \mbox{\rm for all }i=1,\ldots,m, \vspace{1ex}\\
&w\in \Theta,
\end{array}\right.
\end{equation}
where $\Theta$ is a nonempty closed convex set in $\mathcal H$.
For the hard-margin constrained SVM problem, we assume $\Theta$ has a 
non-empty intersection with the set of points satisfying the inequality 
constraints in~\eqref{SVMP1};  this assumption may be dropped in the next 
subsection when we turn to the soft-margin constrained SVM model.
As mentioned earlier, this new model not only captures the optional projection
step in the \emph{Pegasos} \cite{Pegasos}, but also provides greater control
over the weight vector $w$.

To continue, we apply Lagrangian duality from Subsection 
\ref{sec:Lag_duality_for_cvx} to study the constrained SVM \eqref{SVMP1}. The 
Lagrange function is given by
\begin{equation*}
\mathcal{L}(w, \lambda)=\frac{1}{2}\|w\|^2+\sum_{i=1}^m \lambda_i (1-y_i\la w, \mathbf x_i\ra), \; w\in \Theta, \; \lambda\in \R^m.
\end{equation*}
The Lagrange dual function is
\begin{equation*}
\begin{aligned}
    \Hat{\mathcal{L}}(\lambda)=\inf_{w\in \Theta}\mathcal{L}(w, \lambda), \; \lambda\in \R^m_+.
      \end{aligned}
\end{equation*}
Solving the optimality condition of this problem
\begin{equation*}
    0\in  \nabla_w\mathcal{L}(w, \lambda)+N(w; \Theta)=w-\sum_{i=1}^m\lambda_i y_i\mathbf x_i+N(w; \Theta),
\end{equation*}
we obtain (see \cite[Proposition 1.69]{bmn} which holds in Hilbert spaces)
\begin{equation}\label{PDR}
    w = w(\lambda) = \mathcal P\Big(\sum_{i=1}^m \lambda_iy_i\mathbf x_i; \Theta \Big).
\end{equation}
Thus, the Lagrange dual function is given by
\begin{equation*}
\arraycolsep=0.2em
\begin{array}{lcl}
 \Hat{\mathcal{L}}(\lambda) &= & \mathcal{L}(w(\lambda), \lambda) \vspace{1ex}\\
 &= & \frac{1}{2}\|w(\lambda)\|^2 - \sum_{i=1}^m\lambda_i y_i\la\mathbf{x}_i, w(\lambda)\ra + \sum_{i=1}^m\lambda_i \vspace{1ex}\\
 & = & \frac{1}{2}\| \sum_{i=1}^m\lambda_iy_i\mathbf{x}_i - w(\lambda)\|^2 - \frac{1}{2}\big\| \sum_{i=1}^m\lambda_iy_i\mathbf{x}_i \big\|^2 + \sum_{i=1}^m\lambda_i \vspace{1ex}\\
 & = & - \frac{1}{2} \sum_{i=1}^m\sum_{j=1}^m\lambda_i\lambda_j y_iy_j\la \mathbf{x}_i, \mathbf{x}_j\ra + \sum_{i=1}^m\lambda_i + \frac{1}{2}\Big\| \sum_{i=1}^m\lambda_iy_i\mathbf{x}_i - \mathcal P\Big(\sum_{i=1}^m \lambda_iy_i\mathbf x_i; \Theta \Big) \Big\|^2 \vspace{1ex}\\
 & = & -\frac{1}{2}\sum_{i=1}^m \sum_{j=1}^m \lambda_i\lambda_j y_i y_j\la \mathbf x_i, \mathbf x_j\ra+\sum_{i=1}^m \lambda_i+ \frac{1}{2}\Big ( d\Big(\sum_{i=1}^m \lambda_i y_i \mathbf x_i; \Theta \Big)\Big)^2.
 \end{array}
\end{equation*}
It follows that the Lagrange dual problem  of \eqref{SVMP1} is the 
maximization problem
\begin{equation}\label{SVMD1}
\left\{\begin{array}{ll}
{\displaystyle\max}&\Hat{\mathcal L}(\lambda)= -\frac{1}{2}\sum_{i=1}^m \sum_{j=1}^m \lambda_i\lambda_j y_i y_j\la \mathbf x_i, \mathbf x_j\ra+\sum_{i=1}^m \lambda_i + \frac{1}{2} \Big ( d \Big(\sum_{i=1}^m \lambda_i y_i \mathbf x_i; \Theta \Big)\Big)^2, \vspace{1ex}\\
\mathrm{s.t.} & \lambda=(\lambda_1, \ldots, \lambda_m)\in \R^m_+.
\end{array}\right.
\end{equation}
This problem is equivalent to the minimization problem
\begin{equation}\label{SVMD11}
\left\{\begin{array}{ll}
{\displaystyle\min}&\ph(\lambda)= \frac{1}{2}\sum_{i=1}^m \sum_{j=1}^m \lambda_i\lambda_j y_i y_j\la \mathbf x_i, \mathbf x_j\ra-\sum_{i=1}^m \lambda_i - \frac{1}{2} \Big ( d \Big(\sum_{i=1}^m \lambda_i y_i \mathbf x_i; \Theta \Big)\Big)^2, \vspace{1ex}\\
\mathrm{s.t.} & \lambda=(\lambda_1, \ldots, \lambda_m)\in \R^m_+.
\end{array}\right.
\end{equation}
If we let $\mathbf{e}$ be a vector of all ones in $\R^m$ and
$\mathcal{A}:\R^m\to\mathcal H$ be the linear mapping defined by
\begin{equation*}
    \mathcal{A}(\lambda)=\sum_{i=1}^m \lambda_i y_i \mathbf x_i, \ \; \lambda=(\lambda_1, \ldots, \lambda_m)\in \R^m,
\end{equation*}
we can express~\eqref{SVMD11} in the more compact form
\begin{equation}\label{SVMD12}
\left\{\begin{array}{ll}
{\displaystyle\min}&\ph(\lambda)= \frac{1}{2}\|\mathcal A\lambda\|^2-\la\mathbf{e},\lambda\ra - \frac{1}{2} \big ( d(\mathcal A\lambda; \Theta)\big)^2, \vspace{1ex}\\
\mathrm{s.t.} & \lambda=(\lambda_1, \ldots, \lambda_m)\in \R^m_+.
\end{array}\right.
\end{equation}

\begin{Remark} \label{Dual_grad}{\rm (a) In the case where there is no 
constraint $\Theta$, i.e., $\Theta=\mathcal{H}$, it is obvious that
\begin{equation*}
d\Big(\sum_{i=1}^m \lambda_i y_i \mathbf x_i; \Theta \Big)=0 
\end{equation*}
in \eqref{SVMD1}. Thus, this problem reduces to the well-known Lagrange dual problem of the classical SVM. To the best of our knowledge, the Lagrange dual problem \eqref{SVMD1} has not been introduced in the literature.

(b) The gradient of the dual function $\ph$ in \eqref{SVMD11}  is given by 
\begin{equation*}
\nabla\ph(\lambda) = \mathcal A^*\circ \mathcal A\lambda- \mathbf{e}-\mathcal{A}^*\big(\mathcal A\lambda-\mathcal{P}(\mathcal A\lambda; \Theta)\big)=\mathcal A^*(\mathcal P(\mathcal A\lambda;\Theta))-\mathbf{e}.
\end{equation*}
We can show that $\nabla \ph$ is Lipschitz continuous on $\R^m$. 
This property is important to develop gradient and accelerated gradient-type methods for solving \eqref{SVMD1}, and hence the primal problem \eqref{SVMP1}.}
\end{Remark}

We say that a set $\{a_1, \ldots, a_m\}$ in $\mathcal H$ is positively linearly independent if we have the implication
\begin{equation*}
    \big[\sum_{i=1}^m\lambda_ia_i=\sum_{i=1}^m \gamma_ia_i]\Longrightarrow \big[\lambda_i=\gamma_i\; \mbox{\rm for all }i=1, \ldots, m]
\end{equation*}
whenever $\lambda_i, \gamma_i\geq 0$ for all $i=1, \ldots, m$. Obviously, if $\{a_1, \ldots, a_m\}$ is linearly independent, then this set is positively linearly independent.

In what follows, we use the  separability assumption below for the constrained SVM model:
\begin{equation}\label{SVMS}
    \exists u\in \qri(\Theta)\; \mbox{\rm such that }y_i\la u, \mathbf x_i\ra\geq 1\; \mbox{\rm for all }i=1, \ldots, m.
\end{equation}

\begin{Theorem} \label{SVMA} Consider the primal SVM problem \eqref{SVMP1} and its dual form \eqref{SVMD1} under the separability assumption \eqref{SVMS}. 
Then, we have the following conclusions:
\begin{enumerate}
\item The primal SVM problem has a unique optimal solution.
\item Lagrangian strong duality holds.
\item The dual problem always has an optimal solution.
\item If $\Theta=\mathcal H$, then the dual problem has a unique optimal solution if the set $\{y_1\mathbf x_1, \ldots, y_m\mathbf x_m\}$ is positively linearly independent.
\end{enumerate}
\end{Theorem}
\begin{proof}  (a) Consider the primal SVM problem \eqref{SVMD1}. 
Since the set of feasible solutions is nonempty under \eqref{SVMS} and the objective function $f(w)=\frac{1}{2}\|w\|^2$, $w\in \mathcal H$, is strongly convex, this problem has a unique optimal solution.

(b) Under the separability assumption \eqref{SVMS},  because the functions defining the constraint set are affine, the KKT conditions are satisfied by Theorem \ref{kttgpc}. Using Theorem \ref{lsd1}(a), we see that the Lagrangian strong duality holds. 

(c) Since the KKT conditions are satisfied as stated in (b), by Theorem~\ref{lsd1}(a), the dual problem \eqref{SVMD1} also has an optimal solution. 

(d) Let $\Theta=\mathcal H$, and assume that $\{y_1\mathbf x_1, \ldots, y_m\mathbf x_m\}$ is positively linearly independent. Let $w_0$ be the unique optimal solution to the primal SVM problem in this case. Suppose that $\lambda=(\lambda_1, \ldots, \lambda_m), \gamma=(\gamma_1, \ldots, \gamma_m)\in \R^m_+$ are  solutions to the dual problem \eqref{SVMD1}. By Theorem~\ref{lsd1}(b), we have
\begin{equation*}
    w_0=\sum_{i=1}^m \lambda_i y_i\mathbf x_i=\sum_{i=1}^m \gamma_i y_i\mathbf x_i.
\end{equation*}
Using the positively linear independence assumption, we deduce that $\lambda_i=\gamma_i$ for all $i=1, \ldots, m$, which justifies the uniqueness of an optimal solution to the dual problem \eqref{SVMD1}. 
\end{proof}

\begin{Theorem}\label{SVMA1} Under the same setting as in Theorem~\ref{SVMA}, suppose that $w_0$ is the unique optimal solution to the primal problem \eqref{SVMP1}, and $\lambda=(\lambda_1, \ldots, \lambda_m)\in \R^m_+$ is an optimal solution to the dual problem \eqref{SVMD1}. Then, we have
\begin{enumerate}
\item $w_0=\mathcal P\Big(\sum_{i=1}^m \lambda_i y_i \mathbf x_i; \Theta \Big)$.
\item If $\lambda_i>0$, then $y_i\la w_0, \mathbf x_i\ra =1$.
\end{enumerate}
\end{Theorem}
\begin{proof} Suppose that $w_0$ is the unique optimal solution to the primal problem \eqref{SVMP1} and $\lambda=(\lambda_1, \ldots, \lambda_m)\in \R^m_+$ is an optimal solution to the dual problem \eqref{SVMD1}. 
Applying Theorem \eqref{lsd1}(b), we see that the KKT conditions are satisfied. Then (a) follows from \eqref{PDR}, and we also have $\lambda_i (1-y_i\la w_0, \mathbf x_i\ra)=0$ for all $i=1, \ldots, m$. Conclusion (b) is now straight forward. 
\end{proof}

\begin{Remark}
{\rm Theorem \ref{SVMA1} allows us to reconstruct a solution $w_0$ to the primal problem \eqref{SVMP1} from the dual solution $\lambda$ to the dual problem \eqref{SVMD1}.
Moreover, conclusion $\mathrm{(b)}$ shows us which element $\mathbf{x}_i$ lies on the marginal hyperplane.}
\end{Remark}

\subsection{Soft-margin constrained SVM formulation via Lagrange duality}\label{svma soft margin}

In this subsection, we consider an extension of the \emph{soft-margin SVM} to the case of the constrained SVM model. The problem under consideration is
\begin{equation}\label{SVMP2}
\left\{\begin{array}{ll}
{\displaystyle\min_{ w \in \mathcal{H}, \; \xi_i\in \R_{+}}} & f(w, \xi)=\frac{1}{2}\|w\|^2+\frac{C}{m}\sum_{i=1}^m \xi_i, \vspace{1ex}\\
\mathrm{s.t.} & y_i\la w, \mathbf x_i\ra\geq 1-\xi_i,\text{ for all }i=1, \ldots, m,\vspace{1ex}\\
& \xi_i\geq 0\text{ for all }i=1,\ldots,m,\vspace{1ex}\\
&w\in \Theta,
\end{array}\right.
\end{equation}
where $C>0$ is a given constant. 

We will study this problem in relation to the following constrained minimization problem:
\begin{equation}\label{SVMP21}
\left\{\begin{array}{ll}
{\displaystyle\min_{ w \in \mathcal{H}}} & g(w)=\frac{1}{2}\|w\|^2+\frac{C}{m}\sum_{i=1}^m \max\{0, 1-y_i\la w, \mathbf x_i\ra\}, \vspace{1ex}\\
\mathrm{s.t.} &w\in \Theta.
\end{array}\right.
\end{equation}
The following proposition shows the relation between two problems.

\begin{Proposition}\label{smgp} 
Consider the minimization problems \eqref{SVMP2} and \eqref{SVMP21}. Then, we have the following conclusions:
\begin{enumerate}
    \item Problem \eqref{SVMP21} has a unique optimal solution.
    \item $w_0 \in \mathcal{H}$ is an  optimal solution to \eqref{SVMP21} if and only if there exists  $\bar \xi\in \R^m_+$ such that $(w_0, \bar \xi)$ is an optimal solution to \eqref{SVMP2}. Consequently, problem \eqref{SVMP2} also has an optimal solution.
       \end{enumerate}
  \end{Proposition}
  \begin{proof} (a) Since $g$ in \eqref{SVMP21} is strongly convex and continuous, \eqref{SVMP21} has a unique optimal solution. \\[1ex]
  (b) Suppose that $w_0$ is an optimal solution to \eqref{SVMP21}. Let $\bar \xi_i=\max\{0, 1-y_i\la w_0, \mathbf x_i\ra\}$ for $i=1, \ldots, m$. Obviously, $(w_0, \bar \xi_i)$, where $\bar\xi=(\bar \xi_1, \ldots, \bar \xi_m)$, is a feasible solution to \eqref{SVMP2}. Take any $w\in \Theta$ and $\xi=(\xi_1, \ldots, \xi_m)\in\R^m_+$ such that
  \begin{equation*}
      y_i\la w, \mathbf x_i\ra\geq 1-\xi_i\; \mbox{\rm for all }i=1, \ldots, m.
  \end{equation*}
Then, we have
\begin{equation*}
    \xi_i\geq 1-y_i\la w, \mathbf x_i\ra,
\end{equation*}
which implies that $\xi_i\geq \max\{0, 1-y_i\la w, \mathbf x_i\ra\} \geq 0$ for all $i=1, \ldots, m$. 
Therefore, we can show that
\begin{equation*}
\begin{aligned}
    f(w_0, \bar \xi)&= \frac{1}{2}\|w_0\|^2+\frac{C}{m}\sum_{i=1}^m \max\{0, 1-y_i\la w_0, \mathbf x_i\ra\}\\
   & =g(w_0)\leq g(w)= \frac{1}{2}\|w\|^2+\frac{C}{m}\sum_{i=1}^m \max\{0, 1-y_i\la w, \mathbf x_i\ra\}\\
   &\leq \frac{1}{2}\|w\|^2+\frac{C}{m}\sum_{i=1}^m \xi_i=f(w, \xi).
        \end{aligned}
\end{equation*}
This shows that $(w_0, \bar \xi)$ is an optimal solution to \eqref{SVMP2}.\\[1ex]
Now, suppose $(w_0, \bar \xi)$ is optimal solution to 
\eqref{SVMP2}, and let $\widetilde\xi_i\in\R^m_+$ be given by
$\widetilde\xi_i=\max\{0,1-y_i\la w_0,\mathbf x_i\ra\}$. Since
$y_i\la w_0,\mathbf x_i\ra\geq 1-\bar\xi_i$, or
$1-y_i\langle w_0,\mathbf x_i\rangle\leq\bar\xi_i$, and $0\leq\bar\xi_i$ for 
all $i=1,\ldots,m$, we have
\begin{equation*}
\widetilde\xi_i=\max\{0,1-y_i\la w_0,\mathbf x_i\ra\}\leq\bar\xi_i\textrm{ for 
all }i=1,\ldots,m.
\end{equation*}
From this inequality, we get
$\sum_{i=1}^m\widetilde\xi_i\leq\sum_{i=1}^m\bar\xi_i$, and thus
\begin{equation}\label{g_w0}
g(w_0)=\frac{1}{2}\|w_0\|^2+\frac{C}{m}\sum_{i=1}^m\widetilde\xi_i\leq
\frac{1}{2}\|w_0\|^2+\frac{C}{m}\sum_{i=1}^m\bar\xi_i=f(w_0,\bar\xi).
\end{equation}

Now take any $w\in\Theta$ and let $\xi\in\R^m_+$ be given by
$\xi_i=\max\{0,1-y_i\langle w,\mathbf x_i\rangle\}$. With $(w,\xi)$ being 
feasible for \eqref{SVMP2} and $(w_0,\bar\xi)$ optimal, \eqref{g_w0} gives
\begin{equation*}
g(w_0)\leq f(w_0,\bar\xi)\leq f(w,\xi)=\frac{1}{2}\|w\|^2+\frac{C}{m}\sum_{i=1}^m\xi_i=g(w).
\end{equation*}
Thereofore, we conclude that $w_0$ is an optimal solution to 
\eqref{SVMP21}.
\end{proof}

Now, we consider the following maximization problem:
\begin{equation}\label{SVMD2}
\left\{\begin{array}{ll}
{\displaystyle\max_{ \lambda \in \R^m}} & \Hat{\mathcal{L}}_1(\lambda)=-\frac{1}{2}\sum_{i=1}^m \sum_{j=1}^m \lambda_i\lambda_j y_i y_j\la \mathbf x_i, \mathbf x_i\ra+\sum_{i=1}^m \lambda_i+ \frac{1}{2}\Big ( d\Big(\sum_{i=1}^m \lambda_i y_i \mathbf x_i; \Theta \Big)\Big)^2, \vspace{1ex}\\
\mathrm{s.t.} &\lambda_i \in \left[0, \frac{C}{m} \right] \ \; \mbox{\rm for all }  i=1,\ldots, m.
\end{array}\right.
\end{equation}

\begin{Theorem}\label{SVM XI BOUNDS} Consider problems \eqref{SVMP2} and 
\eqref{SVMD2}.   Then, we have the following conclusions:
\begin{enumerate}
\item Each problem has a nonempty optimal solution set. 
\item These two problems have the same optimal value.
\item If $(w_0, \bar \xi)$, where $\bar \xi\in \R^m_+$, is an optimal 
solution to \eqref{SVMP2}, and $\lambda=(\lambda_1, \ldots, \lambda_m)$ is an 
optimal solution to \eqref{SVMD2}, then
\begin{equation}\label{PRJS}
    w_0=\mathcal{P}\Big(\sum_{i=1}^m \lambda_i y_i\mathbf x_i; \Theta\Big).
\end{equation}
 Furthermore, if $0<\lambda_i<\frac{C}{m}$, then $y_i\la\bar  w, \mathbf x_i\ra=1$.
\end{enumerate}

\end{Theorem}
\begin{proof} 
(a) By Proposition~\ref{smgp},  the optimization problem \eqref{SVMP2} has an optimal solution. 
Since the constraint set in \eqref{SVMD2} is a nonempty compact set, this problem has an optimal solution due to the continuity of the objective function.

(b) Let $p$ be the optimal value of \eqref{SVMP2}, and let $d$ be the optimal value of \eqref{SVMD2}. The Lagrange function associated with \eqref{SVMP2} is
\begin{equation*}
\begin{aligned}
    \mathcal{L}(w, \xi, \lambda, \mu)&=\frac{1}{2}\|w\|^2+\sum_{i=1}^m \frac{C}{m}\xi_i+\sum_{i=1}^m \lambda_i (1-\xi_i-y_i\la w, \mathbf x_i\ra)+\sum_{i=1}^m\mu_i(-\xi_i)\\
    &=\frac{1}{2}\|w\|^2+\sum_{i=1}^m\Big(\frac{C}{m}-\lambda_i-\mu_i\Big)\xi_i -\sum_{i=1}^m \lambda_i y_i\la w, \mathbf x_i\ra,
    \end{aligned}
\end{equation*}
where $(w, \xi)\in \mathcal H\times \R^m$ and $(\lambda, \mu)\in \R^m_+\times\R^m_+$. Given any $(\lambda, \mu)\in \R^m_+\times \R^m_+$, we have
\begin{equation*}
    \Hat{\mathcal{L}}(\lambda, \mu)=\inf_{(w, \xi)\in \Theta\times \R^m} \mathcal{L}(w, \xi, \lambda, \mu).
\end{equation*}
We observe that if there exists $i=1, \ldots, m$ such that $\frac{C}{m}-\lambda_i-\mu_i\neq 0$, i.e. $\lambda_i+\mu_i\neq \frac{C}{m}$, then $\Hat{\mathcal{L}}(\lambda, \mu)=-\infty$ by letting $\xi_i\to\infty$ or $-\infty$. 

Now, consider the case where $\lambda_i+\mu_i=\frac{C}{m}$ for all $i=1, \ldots, m$. 
In this case, we use the optimality conditions:
\begin{equation*}
   0\in  \nabla_w\mathcal{L}(w, \xi, \lambda, \mu)+N(w; \Theta)=w-\sum_{i=1}^m \lambda_i y_i\mathbf x_i +N(w; \Theta) \ \;  \mbox{\rm and }\ \;  \nabla_{\xi}\mathcal{L}(w, \xi, \lambda, \mu)=0.
\end{equation*}
Following the justification before \eqref{SVMD1} gives us
\begin{equation*}
    \Hat{\mathcal{L}}(\lambda, \mu)=-\frac{1}{2}\sum_{i=1}^m \sum_{j=1}^m \lambda_i\lambda_j y_i y_j\la \mathbf x_i, \mathbf x_i\ra+\sum_{i=1}^m \lambda_i+ \frac{1}{2}\Big ( d\Big(\sum_{i=1}^m \lambda_i y_i \mathbf x_i; \Theta \Big)\Big)^2.
\end{equation*}
Obviously, the Slater condition is satisfied for problem \eqref{SVMP2} by using sufficiently large $\xi_i$ to have strict inequalities; see Remark \ref{KKTR}. Thus, by Lagrangian strong duality from Theorem~\ref{LSDE} we have
\begin{equation*}
\begin{aligned}
    &p=\sup_{(\lambda, \mu)\in \R^m_+\times \R^m_+}\Hat{\mathcal L}(\lambda, \mu)\\
    &=\sup\big\{\Hat{\mathcal{L}}(\lambda, \mu)\; \big|\; (\lambda, \mu)\in \R^m_+\times \R^m_+, \; \lambda_i+\mu_i=\tfrac{C}{m}\; \mbox{\rm for all }i=1, \ldots, m\big\}\\
    &=\sup\big\{-\frac{1}{2}\sum_{i=1}^m \sum_{j=1}^m \lambda_i\lambda_j y_i y_j\la \mathbf x_i, \mathbf x_i\ra+\sum_{i=1}^m \lambda_i+\frac{1}{2}\Big ( d\Big(\sum_{i=1}^m \lambda_i y_i \mathbf x_i; \Theta \Big)\Big)^2\; \big|\; 0\leq \lambda_i \leq \tfrac{C}{m}, \ \forall i = 1,\ldots, m \big\}\\
    &=\sup\big\{\Hat{\mathcal{L}}_1(\lambda)\; \big |\; 0\leq \lambda_i \le \tfrac{C}{m} \; \mbox{\rm for all } i = 1,\ldots, m\big\} =d.
    \end{aligned}
\end{equation*}
(c) Suppose that $(w_0, \bar \xi)$, where $\bar \xi\in \R^m_+$, is an optimal 
solution to \eqref{SVMP2}, and $\lambda=(\lambda_1, \ldots, \lambda_m)$ is an 
optimal solution to \eqref{SVMD2}. By Theorem~\ref{lsd1}, the KKT conditions 
hold. In particular,
\begin{equation*}
0\in \nabla_w\mathcal{L}(w_0, \xi, \lambda, \mu)+N(w_0; \Theta)=w_0-\sum_{i=1}^m \lambda_i y_i\mathbf x_i+N(w_0; \Theta).
\end{equation*}
Thus, we obtain \eqref{PRJS}.

Furthermore, we have the complementary slackness conditions:
\begin{equation*}
    \lambda_i(1-\xi_i-y_i\la w_0, \mathbf x_i\ra)=0\; \mbox{\rm and }\lambda_i(-\xi_i)=0\; \mbox{\rm for all }i=1, \ldots, m.
\end{equation*}
Thus, if $0<\lambda_i<\frac{C}{m}$, then $\xi_i=0$ and hence $y_i\la w_0, \mathbf x_i\ra=1$. 
\end{proof}

\subsection{Regularized hard-margin SVM formulation} via Lagrange duality
In this subsection, we consider  the following regularized hard-margin SVM model:
\begin{equation}\label{SVMP1Q}
\left\{\begin{array}{ll}
{\displaystyle\min_{ w \in \mathcal{H}}} & \frac{1}{2}\|w\|^2 + h(w), \vspace{1ex}\\
\mathrm{s.t.} & y_i(\la \mathbf x_i,w\ra)\geq 1\; \mbox{\rm for }i=1,\ldots,m,
\end{array}\right.
\end{equation}
where $h \colon  \mathcal{H} \to \oR$ is an l.s.c. proper convex function. In particular, if $h = \delta_{\Theta}$, the indicator of a nonempty closed convex set $\Theta$ in $\mathcal{H}$, then \eqref{SVMP1Q} reduces to the constrained SVM \eqref{SVMP1}. 

Recall that for a function $h\colon \mathcal H\to \oR$, the proximal mapping of $h$ is the set-valued mapping $\prox_{h}\colon \mathcal H\tto \mathcal H$ defined by
\begin{equation*}
\prox_{h}(w)=\argmin\big\{h(y)+\Ts\frac{1}{2}\|w-y\|^2\;\big|\;y\in \mathcal H\big\}.
\end{equation*}
Note that if $h\colon \mathcal H\to \oR$ is an l.s.c. proper convex function, then $\prox_{h}$ is single-valued. The Moreau envelope of $h$ is the associated optimal value function, i.e., 
\begin{equation*}
M_{h}(w)=\inf\big\{h(y)+\Ts\frac{1}{2}\|w-y\|^2\;\big|\;y\in \mathcal H\big\}, \ \; w\in \mathcal H.
\end{equation*}

The Lagrange function is given by
\begin{equation*}
\mathcal{L}(w, \lambda)=\frac{1}{2}\|w\|^2 + h(w) +\sum_{i=1}^m \lambda_i (1-y_i\la w, \mathbf x_i\ra), \; w\in \mathcal H, \; \lambda\in \R^m_+.
\end{equation*}
The Lagrange dual function is
\begin{equation*}
\begin{aligned}
    \Hat{\mathcal{L}}(\lambda)=\inf_{w\in \mathcal{H}}\mathcal{L}(w, \lambda), \; \lambda\in \R^m_+.
      \end{aligned}
\end{equation*}
Solving the inclusion
\begin{equation*}
    0\in  \nabla_w\mathcal{L}(w, \lambda)+ \partial{h}(w) =w-\sum_{i=1}^m\lambda_i y_i\mathbf x_i + \partial{h}(w),
\end{equation*}
we obtain
\begin{equation*}
    w= \mathrm{prox}_h(z), \ \;  \textrm{where} \;  z = \sum_{i=1}^m \lambda_iy_i\mathbf x_i.
\end{equation*}

Thus, the Lagrange dual function is given by
\begin{equation*}
 \Hat{\mathcal{L}}(\lambda)=-\frac{1}{2}\sum_{i=1}^m \sum_{j=1}^m \lambda_i\lambda_j y_i y_j\la \mathbf x_i, \mathbf x_j\ra+\sum_{i=1}^m \lambda_i + \frac{1}{2}\Vert \mathrm{prox}_h(z) - z \Vert^2 + h(\mathrm{prox}_h(z) ).
\end{equation*}
Let us denote
\begin{equation*}
M_h(z) = h(\mathrm{prox}_h(z)) + \frac{1}{2}\Vert \mathrm{prox}_h(z) - z \Vert^2.
\end{equation*}
Clearly, $M_h$ is the Moreau envelope of $h$, and hence it is $1$-smooth, i.e. its gradient $\nabla{M_h}(z) = z - \mathrm{prox}_h(z)$ is $1$-Lipschitz continuous.

Then, the Lagrange dual problem  of \eqref{SVMP1Q} is equivalent to the minimization problem:
\begin{equation*}
\left\{\begin{array}{ll}
{\displaystyle\min}& \mathcal{D}(\lambda) = \frac{1}{2}\sum_{i=1}^m \sum_{j=1}^m \lambda_i\lambda_j y_i y_j\la \mathbf x_i, \mathbf x_j\ra-\sum_{i=1}^m \lambda_i - M_{h}\Big(\sum_{i=1}^m \lambda_iy_i\mathbf x_i\Big), \vspace{1ex}\\
\mathrm{s.t.} & \lambda=(\lambda_1, \ldots, \lambda_m)\in \R^m_+.
\end{array}\right.
\end{equation*}

Let us consider 
\begin{equation*}
D_h(z) = \frac{1}{2}\Vert z \Vert^2 - h(\mathrm{prox}_h(z)) - \frac{1}{2}\Vert \mathrm{prox}_h(z) - z \Vert^2.
\end{equation*}
Then, it is straightforward to show that $D_h$ is convex and $\nabla{D_h}(z) = \mathrm{prox}_h(z)$, which is $1$-Lipschitz continuous.
Since the dual function $\mathcal{D}$ in \eqref{SVMD1} is $\mathcal{D}(\lambda) = D_h\big(\sum_{i=1}^m\lambda_i y_i\mathbf x_i\big) - \sum_{i=1}^m\lambda_i$, it is also convex. 
Moreover, its gradient $\nabla\mathcal{D}(\lambda) = [y_1\mathbf{x}_1\mathrm{prox}_h(z), \ldots, y_m\mathbf{x}_m\mathrm{prox}_h(z)] - \mathbf{e}$ is $L_D$-Lipschitz continuous with $L_D = \lambda_{\max}(A^TA)$, where $A$ is the matrix formed by $y_1\mathbf{x}_1, \ldots, y_m\mathbf{x}_m$, and $\mathbf{e}$ is the vector of all ones.

\subsection{Subgradient methods for solving constrained SVM problems}\label{SMCP}
Since the primal problem \eqref{SVMP21} is convex but nonsmooth, we present a subgradient method to solve it.
To make our analysis transparent, consider the  constrained problem formulated as
\begin{equation}\label{e:constraint}
\left\{\begin{array}{ll}
{\displaystyle\min_{ w \in \mathcal{H}}} & f(w) = f_0(w) + R(w), \vspace{1ex}\\
\mathrm{s.t.} &w\in \Theta,
\end{array}\right.
\end{equation}
where the function $f_0\colon \mathcal H\to\R$ is convex and possibly nonsmooth but subdifferentiable on its domain, the function $R \colon \mathcal{H}\to\R$ is $\gamma$-strongly convex and $L$-smooth (i.e., its gradient is $L$-Lipschitz continuous), and $\Theta$ is a nonempty closed convex set. 

Let $\{\alpha_k\}$  be a sequence of positive numbers. The projected subgradient algorithm is presented as follows. Given a starting point $w_1\in \Theta$, consider the {\em iterative procedure} constructed by
\begin{equation}\label{e:projected sequence}
w_{k+1}=\mathcal P(w_k-\al_k (v_k + \nabla{R}(w_k)); \Theta)\;\mbox{ where }\;v_k\in\partial f_0(w_k),\; k\in\N.
\end{equation}

We also define
\begin{equation}\label{pvk}
    V_k=\min\{f(w_1), \ldots, f(w_k)\}, \; \ k\in \N.
\end{equation}

Recall that a function $f\colon\mathcal  H \to \R$ is said to be strongly convex with parameter $\gamma>0$ if the function $\psi\colon \mathcal H\to \R$ given by
\begin{equation*}
    \psi(w)=f(w)-\frac{\gamma}{2}\|w\|^2, \ \; w\in H,
\end{equation*}
is convex.   

The following two propositions are well known. However, we provide the detailed proof for the convenience of the reader.
\begin{Proposition}\label{scs} Let $f\colon \mathcal H\to \R$ be a strongly convex function with parameter $\gamma>0$. 
Then for any $v\in \partial f(w_0)$ we have
\begin{equation}\label{sconvexsub}
    \la v, w-w_0\ra +\frac{\gamma}{2}\|w-w_0\|^2\leq f(w)-f(w_0)\; \mbox{\rm for all }w\in \mathcal H.
\end{equation}
\end{Proposition}

\begin{proof} 
Since $f$ is strongly convex with parameter $\gamma>0$, there exists a convex function $\psi \colon \mathcal H\to \R$ such that
\begin{equation*}
    f(w) = \psi(w)+\frac{\gamma}{2}\|w\|^2\ \; \mbox{\rm for all }w\in \mathcal H.
\end{equation*}
Fix any $v\in \partial f(w_0)$, by the subdifferential sum rule (see, e.g., \cite{bmn2022}), we have the representation
\begin{equation*}
    v=v_1+\gamma w_0, \; \ \mbox{\rm where }v_1\in \partial \psi(w_0).
\end{equation*}
It follows that
\begin{equation*}
\begin{aligned}
    \la v, w-w_0\ra &=\la v_1, w-w_0\ra+\gamma \la w_0, w-w_0\ra\\
    &\leq \psi(w)- \psi(w_0)+\gamma \la w_0, w-w_0\ra\\
    &=(\psi(w)+\frac{\gamma}{2}\|w\|^2)-(\psi(w_0)+\frac{\gamma}{2}\|w_0\|^2)-\frac{\gamma}{2}\|w-w_0\|^2\\
    &=f(w) - f(w_0)-\frac{\gamma}{2}\|w-w_0\|^2.
    \end{aligned}
\end{equation*}
This implies \eqref{sconvexsub}. 
\end{proof}

\begin{Lemma}\label{HNE} For any $r \geq 0$, we have the upper estimate
\begin{equation}\label{Hn}
    \sum_{i=1}^k \frac{1}{i+r}\leq \ln(k + r + 1)\ \; \mbox{\rm for all }k\in \N.
    \end{equation}
\end{Lemma}
\begin{proof} Since $r\geq 0$, we have
\begin{equation*}
    \int_1^{k+1} \frac{1}{(x+r)}\; dx=\ln(k+r+1) - \ln(r+1) \leq \ln(k+r+1).
\end{equation*}
By the fact that the function $f(x)=\frac{1}{x+r}$ is decreasing on $(0, \infty)$ we have   
\begin{equation*}
    \int_1^{k+1} \frac{1}{x+r}\; dx=\sum_{i=1}^{k}\int_{i}^{i+1}\frac{1}{x+r}\; dx\geq \sum_{i=1}^{k}\frac{1}{i+r}.
\end{equation*}
Combining both inequalities, we obtain \eqref{Hn}. 
\end{proof}

Recall that a function $\psi\colon \mathcal H\to \R$ is said to be $L$-smooth, where $L\geq 0$, if $\psi$ is Fr\'echet differentiable on $\mathcal H $ and its gradient is Lipschitz continuous with constant $L$, i.e.,
\begin{equation*}
    \|\nabla \psi (x)-\nabla \psi (y)\|\leq L\|x-y\|\ \; \mbox{\rm for all }x, y\in \mathcal H.
\end{equation*}

The theorem below discusses the convergence of the projected subgradient method for \eqref{e:constraint}.

\begin{Theorem} \label{SGMC}
Consider the constrained convex optimization problem \eqref{e:constraint}.
Suppose that $f$ is $\gamma$-strongly convex with $\gamma > 0$, $R$ is $L$-smooth,  and there exists $M > 0$ such that 
\begin{equation*}
    \sup\big\{ \| v \| \; \big |\; v\in\partial f_0(w), w \in \Theta \big\} \leq M.
    \end{equation*}
Let $\{x_k\}$ and $\{V_k\}$ be two sequences given by  \eqref{e:projected sequence} and \eqref{pvk}, respectively. Consider the sequence $\{\alpha_k\}$ of stepsizes given by  $\alpha_k= \frac{2}{\gamma (k+r)}$, where  $r \geq \frac{16L^2}{\gamma^2}$,  for $k\in \N$. 
Define 
\begin{equation*}
    u_k=\frac{1}{k}\sum_{i=1}^k x_i, \ \; k\in \N.
\end{equation*}
Then, we have
\begin{equation}\label{Subgrad_conv1}
    f(u_k)- f^{\star} \leq \frac{\gamma r}{4 k}d(w_1;S)^2 + \frac{\ell^2\ln(k+r+1)}{\gamma k} \ \ \mbox{\rm for all }k\in \N,
\end{equation}
where $\ell^2 = 4\|\nabla{R}(w^{*})\|^2 + 2M^2$ and $f^{\star} = \inf_{w\in \Omega}f(w)$.

Furthermore, we also have 
\begin{equation*}
  0\leq   V_k- f^{\star} \leq \frac{\gamma r}{4 k}d(w_1;S)^2 + \frac{\ell^2\ln(k+r+1)}{\gamma k} \ \ \mbox{\rm for all }k\in \N.
\end{equation*}
Consequently, the convergence rate of \eqref{e:projected sequence} using a diminishing step-size $\alpha_k$ is $\mathcal{O}(\ln(k)/k)$.
\end{Theorem}

\begin{proof} 
For simplicity, we only consider the case where $\Theta=\mathcal  H$ and leave the general case for the reader. Observe that the set $S$ of all optimal solutions to \eqref{e:constraint} is nonempty because $f$ is continuous and strongly convex with parameter $\gamma>0$. 
Fix any $w^* \in S$, let $k\in \N$,  and let $d_i = v_i + \nabla{R}(w_i)$. Using the estimate in Proposition~\ref{scs}, we have
\begin{equation*}
\begin{aligned}
    \|w_{i+1}-w^*\|^2& =\|w_i-\alpha_i d_i-w^*\|^2\\  
    &=\|w_i-w^*\|^2-2\alpha_i\la d_i, w_i-w^*\ra+\alpha_i^2\|d_i\|^2\\   
& \leq \|w_i - w^*\|^2-2\alpha_i(f(w_i)-f(w^*)+\frac{\gamma}{2}\|w_i-w^*\|^2)+\alpha_i^2\Vert d_i\Vert^2\\
    &\leq \|w_i-w^*\|^2-2\alpha_i(f(w_i)- f^{\star} +\frac{\gamma}{2}\|w_i-w^*\|^2)+ 2\alpha_i^2\Vert v_i\Vert^2 + 2\alpha_i^2\Vert\nabla{R}(w_i)\Vert^2.
    \end{aligned}
\end{equation*}
By the $L$-smoothness of $R$, we have
\begin{equation*}
\Vert \nabla{R}(w_i)\Vert^2 \leq 2\Vert \nabla{R}(w^*)\Vert^2 + 2\Vert \nabla{R}(w_i) - \nabla{R}(w^*)\Vert^2 \leq 2\Vert \nabla{R}(w^*)\Vert^2 + 2L^2\Vert w_i - w^*\Vert^2.
\end{equation*}
Substituting this inequality into the first one, and using $\|v_i\| \leq M$, we get
\begin{equation*}
\begin{aligned}
    \|w_{i+1}-w^*\|^2 &\leq \|w_i-w^*\|^2-2\alpha_i(f(w_i)- f^{\star}) -\gamma \alpha_i \|w_i-w^*\|^2 + 2\alpha_i^2M^2\\
    &+ 4\alpha_i^2\| \nabla{R}(w^*)\|^2 + 4L^2\alpha_i\| w_i-w^0\|^2 \vspace{1ex}\\
    & = -2\alpha_i(f(w_i)- f^{\star})+ \big(1 - \gamma\alpha_i + 4L^2\alpha_i^2\big)\Vert w_i - w^*\Vert^2 + 2\alpha_i^2\big(2\Vert\nabla{R}(w^*)\Vert^2 + M^2).
    \end{aligned}
\end{equation*}
By the choice of $\alpha_i$, we have $\alpha_i \leq \frac{\gamma}{8L^2}$. Thus, we have the estimate
\begin{equation*}
\begin{aligned}
    \|w_{i+1}-w^*\|^2  &\leq -2\alpha_i(f(w_i)-f^*)+  \big(1 - \frac{\gamma}{2}\alpha_i\big)\Vert w_i - w^*\Vert^2 + \alpha_i^2 \ell^2,
    \end{aligned}
\end{equation*}
where $\ell= \sqrt{ 2\big(2\Vert\nabla{R}(w^*)\Vert^2 + M^2)}$. Since $\alpha_i = \frac{2}{\gamma(i+r)}$, the last inequality leads to
\begin{equation*}
\begin{aligned}
    \frac{4}{\gamma(i+r)}(f(w_i)-f^{\star})  &\leq \Big(1 - \frac{1}{i+r} \Big)\| w_i - w^*\|^2 - \|w_{i+1}-w^*\|^2 + \frac{4}{\gamma^2(i+r)^2} \ell^2.\\
    \end{aligned}
\end{equation*}
Multiplying both sides by $\frac{\gamma(i+r)}{2}$, we obtain
\begin{equation*}
\begin{aligned}
    f(w_i)-f^{\star} &\leq \frac{\gamma(i+r-1)}{4} \| w_i - w^*\|^2 - \frac{\gamma(i+r)}{4} \|w_{i+1}-w^*\|^2 + \frac{ \ell^2}{\gamma(i+r)} .
    \end{aligned}
\end{equation*}
Since $\alpha_i=\frac{2}{\gamma(i+r)} \leq \frac{\gamma}{8L^2}$, we get $i+r \geq \frac{16L^2}{\gamma^2}$. Summing up the above inequalities for  $i=1, \ldots, k$ and using Lemma \ref{HNE}, we get
\begin{equation*}
\begin{aligned}
\sum_{i=1}^k\big( f(w_i) - f^{\star} \big)&\leq \frac{\gamma r}{4} \| w_1 - w^*\|^2 - \frac{\gamma(k+r)}{4} \|w_k - w^*\|^2 + \frac{\ell^2}{\gamma} \sum_{i=1}^k \frac{1}{(i+r)} \\
&\leq \frac{\gamma r}{4} \| w_1 - w^*\|^2+\frac{\ell^2}{\gamma^2} \sum_{i=1}^k \frac{1}{(i+r)}  \leq  \frac{\gamma r}{4} \| w_1 - w^*\|^2+\frac{\ell^2\ln(k+r+1)}{\gamma^2}.
\end{aligned}
\end{equation*}
This inequality leads to
\begin{equation*}
    \frac{1}{k}\sum_{i=1}^k f(w_i)- f^{\star} \leq \frac{\gamma r}{4k} \| w_1 - w^*\|^2+\frac{\ell^2}{\gamma k}\ln(k+r+1).
\end{equation*}
Since $x^*\in S$ is arbitrary, we se that
\begin{equation}\label{SCE1}
\frac{1}{k}\sum_{i=1}^k f(w_i)- f^{\star} \leq \frac{\gamma r}{4k} d(w_1; S)^2+\frac{\ell^2}{\gamma k}\ln(k+r+1).
\end{equation}
Finally, by the convexity of $f$, we have $f(u_k) \leq \frac{1}{k}\sum_{i=1}f(w_i)$.
Substituting this expression into \eqref{SCE1}, we obtain \eqref{Subgrad_conv1}.
We can also deduce from \eqref{SCE1} that
\begin{equation*}
    V_k- f^{\star} \leq \frac{1}{k}\sum_{i=1}^k f(w_i)- f^{\star} \leq \frac{\gamma r}{4k} d(w_1; S)^2+\frac{\ell^2}{\gamma k}\ln(k+r+1).
\end{equation*}
This completes the proof of the theorem.  
\end{proof}

Note that the condition $\sup\{ \Vert g \Vert : g\in\partial{f_0(w)}, w \in \mathcal H \} \leq M$ in Theorem \ref{SGMC} is equivalent to the $M$-Lipschitz continuity of $f_0$ on $\mathcal H$.
However, this assumption can be relaxed by a weaker condition: the boundedness of subgradients on a compact set.
Now, we establish a sufficient condition which guarantees this bounded property of the subgradients.

\begin{Lemma} \label{bdv} Let $f_0 \colon \mathcal H\to \R$ be a continuous convex function. Suppose that $K$ is a nonempty compact set in $\mathcal H$. Then the set
\begin{equation*}
    A=\bigcup_{w\in K}\partial f_0(w)
\end{equation*}
is bounded. 
\end{Lemma}

\begin{proof} Since $f_0$ is convex and continuous, it is locally Lipschitz continuous on $\mathcal H$. Take any $c\in K$ and find $\delta_c>0$ and $\ell_c>0$ such that $f_0$ is Lipschitz continuous on $B(c; \delta_c)$ with Lipschitz constant $\ell_c$. Since $K$ is compact, there exist $c_1, \ldots, c_m\in K$ such that
\begin{equation*}
    K\subset \bigcup_{i=1}^m B(c_i; \delta_{c_i}).
\end{equation*}
Let $\ell=\max\{\ell_{c_i}\; |\; i=1, \ldots, m\}$. 
Now, take any $v\in A$. Then there exist $i\in \{1, \ldots, m\}$ and $\hat{w}\in B(c_i; \delta_{c_i})$ such that $v\in \partial f_0(\hat{w})$. Then, for any $w\in B(c_i; \delta_{c_i})$ we get
\begin{equation*}\label{vk_est}
    \la v, w-\hat w\ra\leq f_0(w) - f_0(\hat w)\leq \ell_{c_i}\|w-\hat w \|\leq \ell \|w-\hat w\|.
\end{equation*}
This implies that $\|v\| \leq\ell$, which completes the proof.
\end{proof}

Following a similar proof as Theorem \ref{SGMC}, but now we treat $f = f_0$ (i.e,. without $R$) as a nonsmooth function, we obtain the following theorem, which generalizes the result from \cite{Pegasos} to Hilbert spaces.

\begin{Theorem} \label{SGMC2}
Consider the constrained convex optimization problem \eqref{e:constraint} in which $f$ is strongly convex with the strong convexity parameter $\gamma>0$ and $R=0$.  
Let $\{w_k\}$ and $\{V_k\}$ be two sequences given by  \eqref{e:projected sequence} and \eqref{pvk}, respectively, and $\alpha_k= \frac{1}{\gamma k}$ for $k\in \N$. Assume that there exists a compact set $K$ such that $w_k\in K$ for all $k\in \N$.  Define 
\begin{equation*}
    u_k=\frac{1}{k}\sum_{i=1}^k x_i, \ \; k\in \N.
\end{equation*}
Then, we have
\begin{equation*}
    f(u_k)- f^{\star} \leq \frac{\ell^2(1+\ln(k))}{2\gamma k}\ \; \mbox{\rm for all }k\in \N.
\end{equation*}
Furthermore, we also have 
\begin{equation*}
  0\leq   V_k- f^{\star} \leq \frac{\ell^2(1+\ln(k))}{2\gamma k}\ \; \mbox{\rm for all }k\in \N,
\end{equation*}
where $f^{\star} = \inf_{w\in \Omega}f(w)$. 
\end{Theorem}

Now, we use Proposition \ref{bdv} for the constrained SVM problem. This proposition is applicable in a more general setting, while significantly simplifying the proof of \cite[Theorem~1]{Pegasos}.

\begin{Proposition}\label{SVMbd} Consider the optimization problem \eqref{SVMP21} and define the set
\begin{equation*}
    D=\{y_i\mathbf x_i\; |\; 1\leq i\leq m\}\cup\{0\}.
\end{equation*}
Consider the subgradient method \eqref{e:projected sequence} applied to this problem with $w_1\in K=C\mbox{\rm co}\, D$  and $0\leq \alpha_k\leq 1$ for $k\in \N$. Suppose that one of the following conditions are satisfied:
\begin{enumerate}
    \item $\Theta=\mathcal H$.
    \item $\Theta$ is a compact set.
    \end{enumerate}
Then there is a compact set $K$ such that $w_k\in K$  for all $k\in \N$.
\end{Proposition}
\begin{proof} It suffices to prove the theorem under (a) because the conclusion under (b) is obvious by Proposition \ref{bdv}. Observe that
\begin{equation*}
v_k\in w_k+\frac{C}{m}\sum_{i=1}^m \co\{0, -y_i\mathbf x_i\}=w_k-\frac{C}{m}\sum_{i=1}^m \co\{0, y_i\mathbf x_i\}\subset w_k-C\,\co\, D=w_k-K.
\end{equation*}
Then we have
\begin{equation*}
    w_{k+1}=w_k-\alpha_k v_k\in (1-\alpha_k)w_k +\alpha_k K.
\end{equation*}
Since $K$ is a convex set, by induction we can easily show that $w_k\in K$ for all $k\in \N$. Thus, the conclusion from Lemma~\ref{bdv} follows because the set $K$ is obviously compact.
\end{proof}

Finally,  we obtain the corollary below using the results of Theorems \ref{SGMC} and \ref{SGMC2}.

\begin{Corollary} Consider the constrained SVM \eqref{SVMP21}. Let $\{x_k\}$ and $\{V_k\}$ be two sequences given by  \eqref{e:projected sequence} and \eqref{pvk}, respectively.
\begin{enumerate}
\item Consider the sequence $\{\alpha_k\}$ of stepsizes given by  $\alpha_k= \frac{2}{\gamma (k+r)}$, where  $r \geq \frac{16L^2}{\gamma^2}$,  for $k\in \N$. Then we obtain the estimates in Theorem \ref{SGMC}.
\item Suppose that $\Theta=\mathcal H$ or $\Theta$ is compact. Consider the sequence $\{\alpha_k\}$ of stepsizes given by  $\alpha_k= \frac{1}{\gamma k}$ and the choice of $w_1$ as in Proposition \ref{SVMbd}. Then we obtain the estimates in Theorem \ref{SGMC2}.
\end{enumerate}

\end{Corollary}

\begin{Remark}{\rm It is straightforward to develop stochastic subgradient methods for the constrained SVM \eqref{SVMP21} and analyze their convergence based on \cite{Pegasos}. }
\end{Remark}

\subsection{Solving constrained SVM problems via dual gradient descent method}

We now turn to an implementation where we apply Lagrangian duality to solve 
the constrained SVM problem. We will provide examples of solving both the 
hard-margin \eqref{SVMP1} and the soft-margin \eqref{SVMP2} constrained SVM 
problems, where we have $m$ data points
$\mathbf x_1,\ldots,\mathbf x_m\in\mathcal H$ and $y_i\in \{-1, 1\}$ for
$i=1,\ldots, m$. Recall the Lagrange dual formulation of the hard-margin 
constrained SVM problem \eqref{SVMD12}
\begin{equation}\label{SVMD1 implement}
\left\{\begin{array}{ll}
{\displaystyle\min}&\ph(\lambda)= \frac{1}{2}\|\mathcal A\lambda\|^2-\la\mathbf{e},\lambda\ra - \frac{1}{2} \big ( d(\mathcal A\lambda; \Theta)\big)^2, \vspace{1ex}\\
\mathrm{s.t.} & \lambda=(\lambda_1, \ldots, \lambda_m)\in \R^m_+
\end{array}\right.
\end{equation}
and the soft-margin constrained SVM problem in the form \eqref{SVMD2} per 
Theorem~\ref{SVM XI BOUNDS}
\begin{equation}\label{SVMD2 implement}
\left\{\begin{array}{ll}
{\displaystyle\min}&\ph(\lambda)= \frac{1}{2}\|\mathcal A\lambda\|^2-\la\mathbf{e},\lambda\ra - \frac{1}{2} \big ( d(\mathcal A\lambda; \Theta)\big)^2, \vspace{1ex}\\
\mathrm{s.t.} & \lambda_i\in[0,\frac{C}{m}]\textrm{ for all }i=1,\ldots,m,
\end{array}\right.
\end{equation}
where $\mathcal A:\R^m\to\mathcal H$ is the linear mapping
$\mathcal A\lambda=\sum_{i=1}^m\lambda_i y_i\mathbf{x}_i$, $\mathbf{e}$ is a 
vector of all ones in $\R^m$, $\Theta$ is a nonempty closed convex subset 
of $\mathcal H$, and $C>0$ is a given constant.

For implementation examples, we use synthetic data in $\R^n$ and define a 
matrix $A\in\R^{n\times m}$ whose $j$th column is $y_j\mathbf{x}_j$. With $A$, 
$\mathbf{e}$, and $\Theta$ as specified,~\eqref{SVMD1 implement} becomes
\begin{equation}\label{SVMD1 matrix}
\left\{\begin{array}{ll}
{\displaystyle\min}&\ph(\lambda)= \frac{1}{2}\|A\lambda\|^2-\la\mathbf{e},\lambda\ra - \frac{1}{2} \big( d(A\lambda;\Theta)\big)^2, \vspace{1ex}\\
\mathrm{s.t.} & \lambda=(\lambda_1, \ldots, \lambda_m)\in \R^m_+
\end{array}\right.
\end{equation}
while \eqref{SVMD2 implement} becomes
\begin{equation}\label{SVMD2 matrix}
\left\{\begin{array}{ll}
{\displaystyle\min}&\ph(\lambda)= \frac{1}{2}\|A\lambda\|^2-\la\mathbf{e},\lambda\ra - \frac{1}{2} \big( d(A\lambda;\Theta)\big)^2, \vspace{1ex}\\
\mathrm{s.t.} & \lambda_i=\in[0,\frac{C}{m}]\text{ for all }i=1,\ldots,m.
\end{array}\right.
\end{equation}

One can solve~\eqref{SVMD1 matrix} and \eqref{SVMD2 matrix} using an iterative 
method, e.g., the projected gradient method in~\eqref{e:projected sequence} or 
Nesterov's accelerated gradient method. In our implementation, we use 
constrained convex optimization solvers from the Python libraries CVXOPT 
\cite{cvxopt} and SciPy \cite{scipy}. As shown in Remark~\ref{Dual_grad} (b), 
we have
\begin{equation}\label{dual gradient}
\nabla \ph(\lambda)= A^T\mathcal P(A\lambda;\Theta) - \mathbf{e},
\end{equation}
where $\mathcal P(A\lambda;\Theta)$ is the Euclidean projection of $A\lambda$ 
onto $\Theta$,
and this gradient is $L$-Lipschitz continuous with Lipschitz constant
$L = \lambda_{\max}(A^TA)$, the largest eigenvalue of $A^TA$.

Suppose that we are applying either a standard gradient or accelerated 
gradient method to solve \eqref{SVMD1 matrix}, then the algorithm generates an 
iterative sequence $\{\lambda_k\}\subset\R^m_+$ to approximate a solution 
$\lambda^{*}$ of \eqref{SVMD1 matrix}. Given $\lambda_k$ at the end of the 
algorithm, we recover \begin{equation*}
w_k = \mathcal{P}\big( A\lambda_k;\Theta \big)
\end{equation*}
as an approximate solution to the primal SVM problem \eqref{SVMP1}.
If we can guarantee $\|\lambda_k - \lambda^{*}\| \leq \epsilon$ for a given tolerance $\epsilon > 0$, then we have
\begin{equation*}
\begin{array}{lcl}
\| w_k - w^{*}\| & = & \| \mathcal{P}\big( A\lambda_k;\Theta \big) - \mathcal{P}\big( A\lambda^{*};\Theta \big) \| \vspace{1ex}\\
&\leq & \|A(\lambda_k - \lambda^{*})\| \vspace{1ex} \\
& \leq & \|A\|\|\lambda_k - \lambda^{*}\| \vspace{1ex}\\
& \leq & \|A\|\epsilon.
\end{array}
\end{equation*}
This shows that $w_k$ is also an $\hat{\epsilon}$-solution of \eqref{SVMP1} with $\hat{\epsilon} = \|A\|\epsilon$.
However, for gradient-type methods, we often have the guarantee $\varphi(\lambda_k) - \varphi(\lambda^{*}) \leq \epsilon$.
In this case, we can also prove that $w_k = \mathcal{P}\big( A\lambda_k;\Theta 
\big)$ is an approximate solution to \eqref{SVMP1}, but we omit this analysis.  
Furthermore, in our implementation we regularize $w$ only and
use constraints corresponding to the indices of the non-zero 
optimal Lagrange multipliers to calculate the corresponding bias term per 
Theorem~\ref{SVMA1}~(b) and Theorem~\ref{SVM XI BOUNDS}~(c).


\subsection{Numerical illustrations}\label{subsec:num_examples}
It has been observed in the literature that the dual gradient descent method has better performance than the subgradient method, both theoretically and practically.
In this section, we provide a few illustrative examples to verify this belief for our model and algorithm.

\subsubsection{Dual gradient method for solving unconstrained SVM}
We first apply our dual gradient descent method to solve the unconstrained SVM.
We generate a sample dataset of 1,000 points in $\R^2$ labeled by two classes.  
We first solve the dual problem~\eqref{SVMD1 matrix} without constraint on 
$w$, i.e.,~$\Theta=\R^2$, and use the methods from earlier sections to produce 
optimal weight vector $w^*$ and bias $b^*$ from optimal Lagrange multipliers 
$\lambda^*$. For our unconstrained hard-margin implementation, we used the 
{\tt cp} function from CVXOPT which uses the gradient of the Lagrangian dual 
cost function \eqref{dual gradient} and its Hessian. For the unconstrained 
problem, these are
\begin{equation*}
\nabla\varphi(\lambda)=A^TA\lambda-\mathbf{e}\text{ and }
\nabla^2\varphi(\lambda)=A^TA.
\end{equation*}

Implementing this using our sample data produces the separating hyperplane 
shown in Figure~\ref{fig:plot-unconstrained} with weight vector
$w_U^*=(1.689, -2.446)$ and bias $b_U^*=3.160$. For the unconstrained, case, 
we have $\|w_U^*\|=2.972$ which provides an optimal value 
$\frac{1}{2}\|w_U^*\|^2=4.4170$ for the hard-margin SVM problem \eqref{SVMP1}.
\begin{figure}[ht]
\centering
\includegraphics[width=0.7\linewidth]{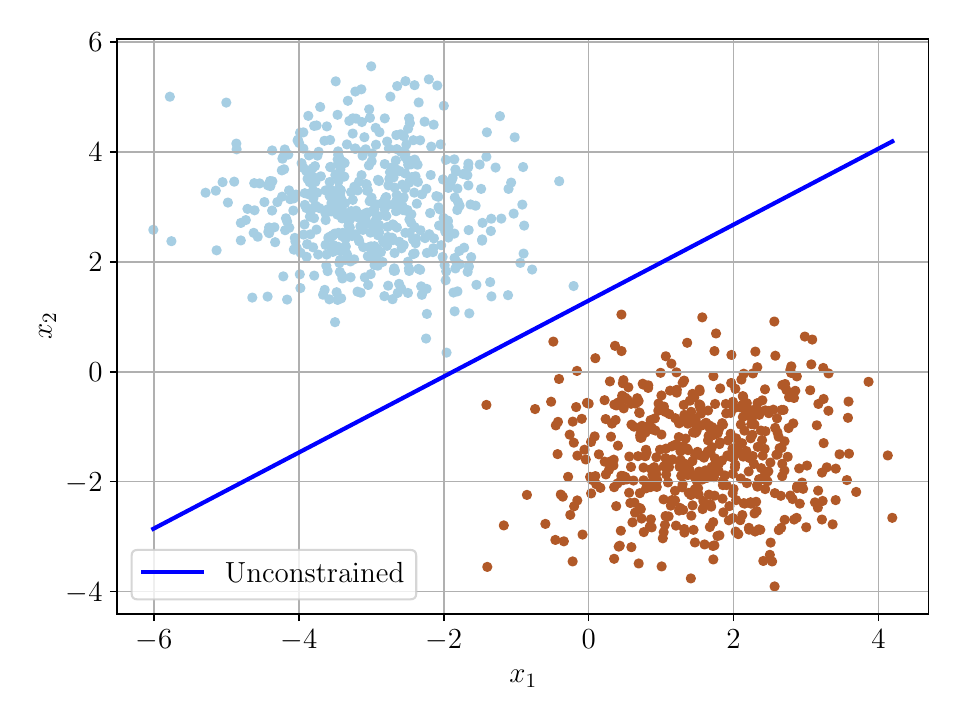}
\caption{Unconstrained solution to SVM problem via Lagrangian duality}\label{fig:plot-unconstrained}
\end{figure}

\clearpage
\subsubsection{Dual gradient descent method for solving SVM with a feasible closed ball constraint}

Our first constrained SVM example simulates a situation where we desire to 
find a weight vector within a given distance $r>0$ to a specified vector
$\bar x$, or $w^*\in\Theta=B^\prime(\bar x;r)$, essentially implementing 
the Pegasos projection using the Lagrangian duality framework, we have 
developed.  We select the constraint vector $\bar x$ and radius $r$ so there 
is a weight vector $w\in\Theta$ and corresponding bias $b$ satisfying the 
hard-margin SVM constraints in problem \eqref{SVMP1}.

Here, we implement the techniques from Subsection \ref{svma hard margin} using 
the first-order constrained optimization method {\tt minimize} from SciPy with 
gradient
$\nabla\varphi(\lambda)=A^T\mathcal P(A\lambda;\Theta)-\mathbf{e}$, where 
$\mathcal P(A\lambda;\Theta)$ is the Euclidean projection of $A\lambda$ onto 
$\Theta$. From the Lagrangian dual solution $\lambda^*$, we obtain the primal 
solution $w^*=A\lambda^*$ and apply complementary slackness from the KKT 
conditions using the constraints corresponding to the indices of the 
non-negative Lagrange multipliers to produce the optimal bias~$b^*$.

Using $\bar x=(2.5,-2.5)$ and $r=0.2$ produces the separating hyperplane shown 
in Figure~\ref{fig:plot-constrained_ball-feasible} produced by weight vector 
$w^*=(2.3455,-2.627)$ and bias $b^*=3.5795$. Margins of 1 are achieved, and in 
this example, where $w^*$ is constrained to have a norm larger than that of 
$w_U^*$, we have an optimal value $\frac{1}{2}\|w^*\|^2=6.2013$ for the 
constrained hard-margin SVM problem \eqref{SVMP1} which naturally exceeds that 
of the unconstrained problem.
\begin{figure}[ht]
\centering
\includegraphics[width=0.7\linewidth]{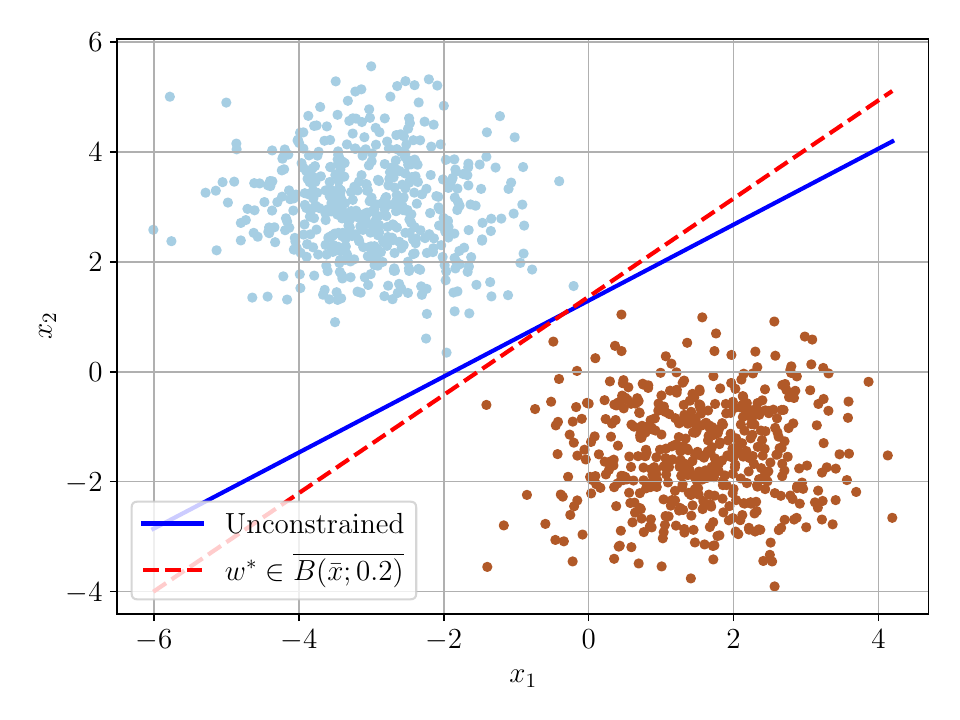}
\caption{Constrained SVM solution with $w^*\in B^\prime((2.5,-2.5);0.2)$\label{fig:plot-constrained_ball-feasible}}
\end{figure}

\clearpage
\subsubsection{Dual gradient descent method for solving SVM with a non-feasible closed ball constraint}

For our next example of a constrained problem, we again implement the Pegasos 
projection step using the Lagrangian duality framework, but this time, we 
constrain the weight vector $w$ using the closed convex set 
$\Theta=B^\prime(0;r)$ choosing $r<\|w_U^*\|$ so the additional 
hard-margin constraints from \eqref{SVMP1} are not attainable using any 
$w\in\Theta$. Therefore, for this example, we must use the constrained 
soft-margin SVM problem \eqref{SVMP2} so the primal problem will be feasible.  
Comparing the resulting Lagrangian dual formulations \eqref{SVMD1 matrix} and 
\eqref{SVMD2 matrix} attained from Subsections \ref{svma hard margin} and
\ref{svma soft margin}, the only adjustment needed for the soft-margin 
implementation is to tighten the constraint on the Lagrange multiplier vector 
$\lambda$ changing it from $\lambda_i\in\R_+$ to $\lambda_i\in[0,\frac{C}{m}]$ 
for all $i=1,\ldots,m$.

As in the previous example, here we also implement the techniques from 
Subsection \ref{svma hard margin} using Scipy's {\tt minimize} function with 
the same expression for the gradient
$\nabla\varphi(\lambda)=A^T\mathcal P(A\lambda;\Theta)-\mathbf{e}$ using the 
Euclidean projection again for a different closed convex set $\Theta$. The optimal bias $b^*$ is also determined in the same manner as the previous example.

Using $r=2.5$ produces the separating hyperplane shown in Figure 
\ref{fig:plot-constrained_ball} produced by weight vector $w^*=(1.421,-2.057)$ 
and bias $b^*=2.598$. As $\|w^*\|=2.5<\|w_U^*\|$, we obtain a smaller optimal 
value $\frac{1}{2}\|w^*\|^2=3.1250$ than the unconstrained problem yet margins 
of only 0.84 are achieved.
\begin{figure}[ht]
\centering
\includegraphics[width=0.7\linewidth]{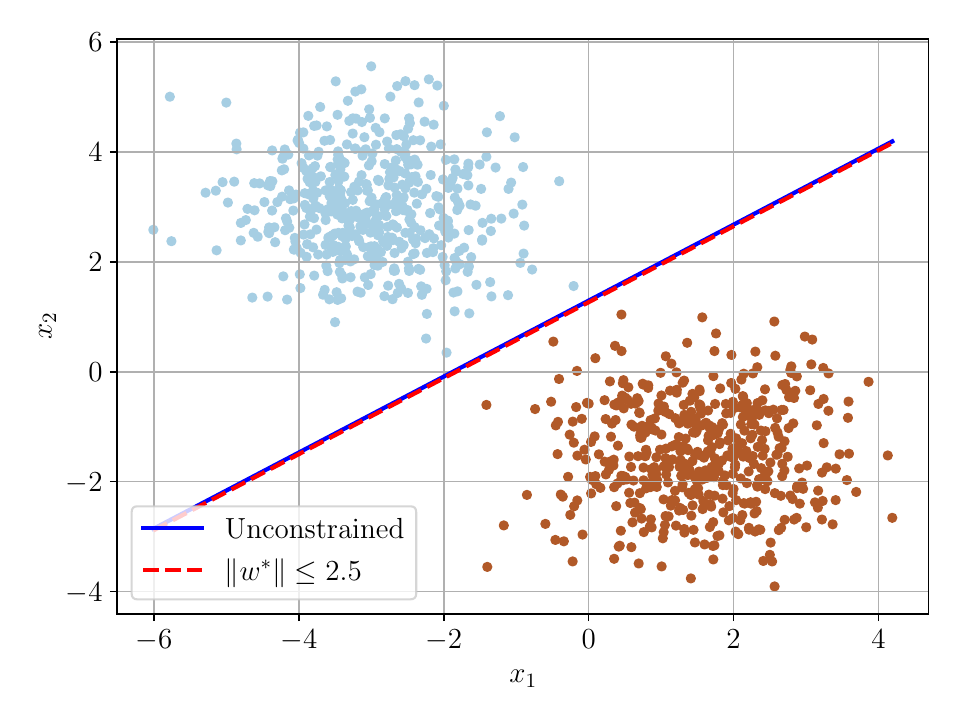}
\caption{Constrained SVM solution with $\|w^*\|\leq 2.5$}\label{fig:plot-constrained_ball}
\end{figure}


\clearpage
\subsubsection{Gradient descent method for solving SVM with the hyperplane constrained to a point}

For the last example, we constrain the separating hyperplane to pass through a 
point $\ox\in\mathbb R^2$. This can be achieved by adding an additional primal 
constraint $\la w,\ox\ra+b=0$ to \eqref{SVMP1} and optimizing over 
$\Theta=\R^2$. This results in a constrained hard-margin primal problem
\begin{equation*}
\left\{\begin{array}{ll}
{\displaystyle\min_{ w \in \R^2}} & \frac{1}{2}\|w\|^2, \vspace{1ex}\\
\mathrm{s.t.} & y_i(\la \mathbf{x}_i,w\ra+b)\geq 1\; \mbox{\rm for all }i=1,\ldots,m, \vspace{1ex}\\
&\la w,\ox\ra + b = 0.
\end{array}\right.
\end{equation*}
Using Lagrange multipliers $\lambda\in\R^m_+$ and $\mu\in\R$ and simplifying 
the subgradient of the Lagrangian at optimality leads to a dual system
\begin{equation*}
\left\{\begin{array}{ll}
{\displaystyle\min}&\ph(\lambda)= \frac{1}{2}\|\bar A\lambda\|^2-\la\mathbf{e},\lambda\ra, \vspace{1ex}\\
\mathrm{s.t.} & \lambda=(\lambda_1, \ldots, \lambda_m)\in \R^m_+,
\end{array}\right.
\end{equation*}
where $\bar A\in\R^{n\times m}$ is composed similarly to $A$ with the data 
matrix $X$ shifted by $\ox$ so the $j$th column of $\bar A$ is 
$y_j(\mathbf{x}_j-\ox)$. For this example, we employ the projected gradient 
method from Subsection~\ref{SMCP} for the dual problem and construct each 
iterate by projecting a step in the direction of the gradient of the cost 
function onto the constraint set.

Letting $\ox=(2,4)$, we achieved an optimal weight vector
$w^*=(3.839, -3.061)$ and bias $b^*=4.566$.  The usual desired margins are 
achievable under this constrained case, and naturally, the norm of the the 
optimal constrained weight vector $\|w^*\|=4.910$ exceeds that of the norm of 
the optimal unconstrained weight vector $\|w_U^*\|=2.972$ since the optimal 
separating hyperplane for the unconstrained problem does not pass through 
$\ox$. In this example, the desired margins of 1 are attained, and the optimal 
value $\frac{1}{2}\|w^*\|^2=12.053$ exceeds that of the unconstrained 
problem.
\begin{figure}[ht]
\centering
\includegraphics[width=0.7\linewidth]{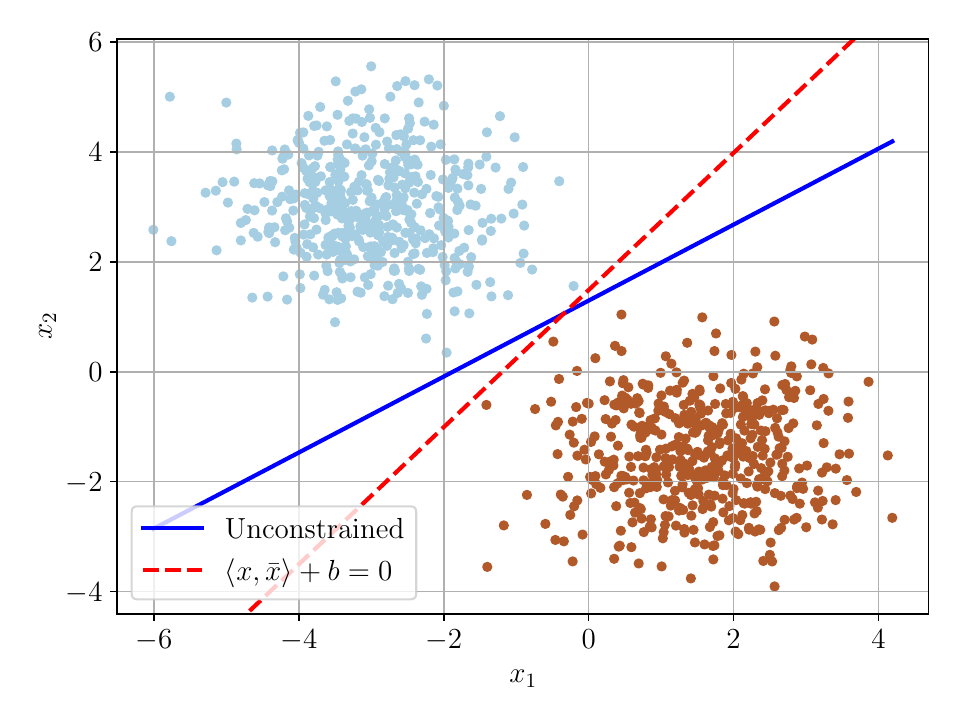}
\caption{Constrained SVM solution through $\ox=(2, 4)$}\label{fig:plot-constrained_xbar}
\end{figure}

\end{document}